\documentclass[11pt]{amsart}
\usepackage{amsmath, amssymb, amsthm}
\usepackage{enumerate}
\usepackage[margin=1in]{geometry}
\usepackage{esint}
\usepackage{xcolor}
\usepackage{hyperref}
\hypersetup{colorlinks = true, linkbordercolor = red}

\usepackage[pdftex]{graphicx}
\pdfoutput=1

\newtheorem{theorem}{Theorem}[section]
\newtheorem{prop}[theorem]{Proposition}
\newtheorem{lemma}[theorem]{Lemma}
\newtheorem{coro}[theorem]{Corollary}
\newtheorem{definition}[theorem]{Definition}
\newtheorem{remark}[theorem]{Remark}
\newtheorem{conj}[theorem]{Conjecture}

\numberwithin{equation}{section}

\def\R{\mathbb{R}}

\def\N{\mathbb{N}}
\def\S{\mathbb{S}}

\def\e{\varepsilon}
\def\eps{\epsilon}

\def\n{\nu}
\def\G{\Gamma}
\def\D{\Delta}
\def\O{\Omega}
\def\a{\alpha}
\def\b{\beta}
\def\g{\gamma}
\def\d{\delta}
\def\s{\sigma}

\def\k{\kappa}
\def\l{\lambda}

\def\o{\omega}
\def\t{\tau}
\def\H{\mathcal{H}}
\def\F{\mathcal{F}}

\def\n{\nabla}

\def\vol{\mathrm{vol}}

\newcommand{\chrc}[1]{\mathrm{1}_{#1}}
\renewcommand{\div}{\text{div}}

\def\de{\partial}

\def\Xint#1{\mathchoice
{\XXint\displaystyle\textstyle{#1}}%
{\XXint\textstyle\scriptstyle{#1}}%
{\XXint\scriptstyle\scriptscriptstyle{#1}}%
{\XXint\scriptscriptstyle\scriptscriptstyle{#1}}%
\!\int}
\def\XXint#1#2#3{{\setbox0=\hbox{$#1{#2#3}{\int}$ }
\vcenter{\hbox{$#2#3$ }}\kern-.6\wd0}}

\def\dashint{\Xint-}

\title[Nondegeneracy and stability]{Nondegeneracy and stability in the limit of a one-phase singular perturbation problem}
\author{Nikola Kamburov}
\address{Nikola Kamburov, Facultad de Matem\'aticas, Pontificia Universidad Cat\'olica de Chile, Avenida Vicu\~na Mackenna 4860, Santiago 7820436, Chile}
\email{nikamburov@mat.uc.cl}

\thanks{The author was partially supported by Proyecto FONDECYT Regular No.\ 1201087.}

\dedicatory{To my teacher David Jerison and to his math ``as the art of the possible"}

\begin{document}

\begin{abstract}
We study solutions to a one-phase singular perturbation problem that arises in combustion theory and that formally approximates the classical one-phase free boundary problem. We introduce a natural density condition on the transition layers themselves that guarantees that the key nondegeneracy growth property of solutions is satisfied and preserved in the limit. We then apply our result to the problem of classifying global stable solutions of the underlying semilinear problem and we show that those have flat level sets in dimensions $n\leq 4$, provided the density condition is fulfilled. The notion of stability that we use is the one with respect to \emph{inner domain deformations} and in the process, we derive succinct new formulas for the first and second inner variations of general functionals of the form $I(v) = \int |\nabla v|^2 + \mathcal{F}(v)$ that hold in a Riemannian manifold setting.
\end{abstract}

\keywords{singular perturbation problem, one-phase free boundary problem, nondegeneracy, second inner variation, stable solutions, rigidity}
\subjclass[2020]{35R35, 35B25, 35B35, 35B65, 35D30}

\bibliographystyle{alpha}
\maketitle

\section{Introduction}

The present paper aims to contribute to the understanding of the limit behaviour of \emph{nonnegative} critical points of the energy functional
\begin{equation}\label{eq:energy}
I_\e(v, \O):= \int_\O \left( |\n v|^2 + \F_\e(v) \right) dx, %\quad v:\O\to [0,\infty),
\end{equation}
in which $\O\subseteq\R^n$ is a domain and the potential $\F_\e(t)$ approximates the characteristic function 
$$\F_0(t):=\chrc{(0,\infty)}(t),$$ 
as $\e\downarrow 0$. Specifically, we will be interested in potentials $\F_\e$ of the form 
\begin{equation}
\F_\e(t):= \begin{cases} \int_0^t 2 f_\e(s)\, ds, & \text{for } t\geq 0, \\
0 & \text{for } t<0,
 \end{cases}
\end{equation}
where $f_\e(t):=\e^{-1} f(t/\e)$ for a given nonnegative function $f\in C^\infty_c([0,\infty))$, satisfying
\begin{align}
& f\geq 0, \quad  \text{supp }f=[0,T], \quad \int_0^T 2 f(s)\,ds = 1, \\
& c_0 s \leq f(s) \leq c_0^{-1} s \quad \text{when } s\in [0,\tau], \label{eq:condf}
\end{align}
for some constants $0<\tau<T<\infty$ and $c_0> 0$. Note that hypothesis \eqref{eq:condf} is simply a quantitative way of expressing $f'(0)>0$ (which can be relaxed -- see the discussion after Theorem \ref{thm:nondeg:main} below). 

For any $\e>0$, nonnegative critical points $u_\e\in H^1_\text{loc}(\O)$ of $I_\e$ solve the semilinear elliptic PDE
\begin{equation}\label{eq:SLe}
\left\{
\begin{aligned}
 u_\e & \geq 0 & \quad \mbox{in } \O, \\
   \Delta  u_\e &= f_\e(u_\e) & \quad \mbox{in } \O,
  \end{aligned}\right.
\end{equation}
in a weak sense. As $f_\e(u_\e) \in L^\infty(\O),$  the Harnack inequality implies that a solution $u_\e$ of \eqref{eq:SLe} must be locally bounded, while because of \eqref{eq:condf}, the strong maximum principle yields that either $u_\e>0$ a.e.\ or $u_\e \equiv 0$ in $\O$. Semilinear elliptic regularity theory then tells us that $u_\e$ is actually a smooth, classical solution of \eqref{eq:SLe}, that is either identically zero, or strictly positive. 

The functionals $I_\e(v, \O)$ formally converge as $\e \downarrow 0$ to the \emph{Alt-Caffarelli energy functional}
\begin{equation}\label{eq:OPfunc}
I_0(v, \O):=\int_\O \left(|\n v|^2 + \F_0(v)\right) dx, \qquad v:\O\to [0,\infty),
\end{equation}
whose associated Euler-Lagrange equations form the classical \emph{one-phase free boundary problem} (FBP)
  \begin{equation}\label{FBP}
\left\{
\begin{aligned}
 u\geq 0 & \quad \mbox{in } \O, \\
   \Delta  u=0 & \quad \mbox{in } \O_0^+(u):=\{x \in \O: u(x)>0\}, \\
   |\nabla u|=1 & \quad \mbox{on }  F_0(u):=\de \O_0^+(u) \cap \O,
  \end{aligned}\right.
  \end{equation}
in which the set $\O_0^+(u)$ is the \emph{positive phase} of $u$, its complement $Z_0(u):=\{x\in \O: u(x)=0\}$ is its \emph{zero phase}, while the abrupt interface $F_0(u)$, caused by the discontinuity of $\F_0$,  is known as the \emph{free boundary}.

The energy functional $I_\e$ appears in models of flame propagation (\cite{laminarflames}) and there has been substantial literature devoted to understanding the underlying \emph{singular perturbation problem} \eqref{eq:SLe} and its parabolic counterparts (we refer the reader to \cite{BereCafNiren1990, CafVasquez1995, CafLW97-1, Weiss2003singperturbation, LedWolanski2006, Karakhanyan2020} and references therein). Of particular interest has been exploring the sense in which critical points $u_\e$ of $I_\e$ and their transition layers $\{\theta \e \leq u_\e \leq T\e\}$, $\theta\in (0,T]$, converge to solutions $u$ of \eqref{FBP} and their free boundaries $F_0(u)$, respectively, and how regular the latter are. 

The case of nonnegative critical points $u_\e$ that locally mimimize the energy $I_\e$ was studied in detail in the book by Caffarelli and Salsa \cite{CafSalsa}. The analysis of the interface convergence as well as the preliminary, measure-theoretic regularity of the resulting free boundary rests on the two fundamental estimates of \emph{uniform Lipschitz continuity} (see Proposition \ref{prop:Lip} below) and \emph{uniform nondegeneracy}. The latter precisely states  (see \cite[Theorem 1.8]{CafSalsa} or \cite[Lemma 4.2]{AS}) that at a distance $r\geq \l \e$ away from points $x\in \O$, where $u_\e(x)\geq \theta \e$ for a fixed $\theta\in (0,T]$, the solution grows to be at least a multiple of $r$:
\begin{equation}\label{eq:nondeg0}
\sup_{B_r(x)} u_\e \geq c r, %\quad \text{as long as } B_r(x)\subset \O,
\end{equation}
for some constants $c, \l >0$. The nondegeneracy property underpins the local Hausdorff distance convergence of the superlevel sets $\{u_\e \geq \theta \e\}$ to the positive phase $\O^+_0(u)$ of the limit $u$. Being passed down to $u$,
\begin{equation}\label{eq:nondegFB0}
\sup_{B_r(x)} u \geq c r, \quad \text{for all } x\in \overline{\O_0^+(u)} \quad \text{and all } B_r(x)\subset \O,
\end{equation}
it is then instrumental in the blow-up analysis that explores the regularity of the free boundary $F_0(u)$ and the sense in which $u$ solves \eqref{FBP}. At this stage, there is another key basic estimate at play: the \emph{positive density} of the zero phase $Z_0(u)$, which states that 
\begin{equation}\label{eq:posdens0}
|Z_0(u)\cap B_r(x)| \geq \k |B_r| \quad \text{for all } x\in Z_0(u), \text{ and all } B_r(x)\subset \O,
\end{equation}  
for some constant $\k>0$. The positive density estimate \eqref{eq:posdens0} is essential in ruling out the possibility of a blowup limit $u_0$ of $u$ that is of \emph{wedge} type: $u_0(x)=s|x_n|$ for some $s>0$, which is a vestige of a singularity in $F_0(u)$. %Similar estimates were established earlier for the class of so-called \emph{minimal} solutions to \eqref{eq:SLe} (see \cite[Theorem 4.1]{BereCafNiren1990}.

Whereas the uniform Lipschitz continuity continues to hold for solutions of \eqref{eq:SLe} that are not necessarily energy minimizing, the nondegeneracy property does not and neither is valid the positive zero-phase density in the limit. This is illustrated by the family of one-dimensional, wedge-like, solutions $\{V_{\e}^s(t)\}_{\e>0, s\in (0,1)}$, given by the unique solutions to the ODE problem:
\begin{equation}
(V_{\e}^s)'' = f_\e(V_{\e}^s) \quad \text{in } \R, \quad \text{with}\quad (V_{\e}^s)'(0) = 0 \quad \text{and} \quad \lim_{t\to\pm \infty} (V_{\e}^s)'(t) = \pm s.
\end{equation}
which blow down to the wedge $s|t|$ for slopes $s\in (0,1)$ (see \cite[Proposition 3.1]{LedWolanski2006CPDE} or \cite[Section 2.3]{FR-RO2019}). As $V_\e^s(0)\in (0, T\e)$, these solutions have a nontrivial interface region, and the 1D solutions of \eqref{eq:SLe} in $\R^n$, given by
$
u_\e(x)=V_\e^\e(x_n),
$
certainly fail the uniform nondegeneracy estimate \eqref{eq:nondeg0}, since they tend to $0$, as $\e \downarrow 0$.

The case of general critical points of $I_\e$ (nonnegative as well as sign-changing) was studied in depth in a series of papers by Lederman and Wolanski \cite{LedWolanski1998, LedWolanski2006CPDE, LedWolanski2006}. For the one-phase singular perturbation scenario, the authors showed that $u_\e$ converge locally uniformly to a limit $u$ which is harmonic in $\O_0^+(u)$ and which satisfies the free boundary gradient condition in viscosity sense (see Definition \ref{def:visco}) as well as pointwise  at regular points of $F_0(u)$, \emph{provided} the limit $u$ satisfies the \emph{nondegeneracy condition} \eqref{eq:nondeg0}. Assuming additionally the \emph{positive density} condition \eqref{eq:posdens0} on the zero phase of $u$, they then obtained that the free boundary $F_0(u)$ is a smooth hypersurface, except on a relatively closed subset of $(n-1)$ Hausdorff measure zero.%, employing the Alt-Caffarelly theory \cite{AC}.

We would like to emphasize that in the cited results above, the additional hypotheses leading to a good regularity theory are made on the limit $u$, and \emph{not} on the critical points $u_\e$ of $I_\e$. The first objective of our paper is to identify a natural condition on the solutions $u_\e$ of \eqref{eq:SLe} themselves that guarantees that the limit $u$ will inherit both key properties \eqref{eq:nondegFB0}-\eqref{eq:posdens0}. We achieve it by introducing the notion of $\mathcal{D}(\k, L)$ \emph{density property} of the interface of $u_\e$. Denote by
\begin{equation}\label{eq:defZF}
Z_\e^\theta(u_\e):=\{x\in \O: u_\e(x)\leq \theta\e \} \quad \text{and} \quad F_\e^\theta(u_\e):= \{x\in \O: \theta\e\leq u_\e(x)\leq T\e\}
\end{equation}
the two parts of the transition region $\{u_\e \leq T\e\}$, divided by the level set $\{u_\e=\theta \e\}$, for $\theta \in (0, T]$. 
\begin{definition}\label{def:DKL}
We will say that (the interface of) $u_\e$ satisfies the density property $\mathcal{D}(\k, L)$ in $\O$ for some $\k \in (0,1]$ and $L>0$ if %for every $x\in F_\e^\tau(u_\e)$ and every $r\geq L\e$ such that $$, it is the case that
\begin{equation}\label{def:DKL:eq}
|Z_\e^{\tau/4}(u_\e)\cap B_{r/2}(x)| \geq \k |B_{r/2}| \quad \text{for all } x\in F_\e^\tau(u_\e) \text{ and all } r\geq L\e, \text{ such that } B_{r}(x)\subseteq \O.
\end{equation}
Here $\tau$ refers to the constant in \eqref{eq:condf}.
\end{definition}
The condition \eqref{def:DKL:eq} is a natural one that minimizers of $I_\e$, in particular, fulfill for universal positive constants $\k, L$ (see Proposition \ref{prop:minDKL}). It is not difficult to envision why the limit $u$ of solutions $u_\e$ of \eqref{eq:SLe}, which satisfy a $\mathcal{D}(\k,L)$ density property uniformly as $\e \downarrow 0$, will inherit the positive density \eqref{eq:posdens0} of the zero phase $Z_0(u)$. What is less obvious is that this property actually guarantees that the $u_\e$ satisfy the uniform nondegeneracy bound \eqref{eq:nondeg0}. This is the content of our first main result. 
\begin{theorem}\label{thm:nondeg:main}
Let $\k \in (0,1]$, $L>0$ and $\theta\in (0,\tau]$. There exist positive constants $c$ and $\e_0$, depending on $\k, L, n$ and $f$, and a constant $M>0$, depending on $\theta, n$ and $f$, such that if $\e\leq \e_0$ and $u_\e \in C^2(B_{2})$ is a solution of \eqref{eq:SLe} in $B_{2}$, for which
\begin{enumerate}
\item $u_\e(0) \leq T \e$, and
\item the interface of $u_\e$ satisfies the density property $\mathcal{D}(\k, L)$ in $B_{1}$,  %for some $\k\in (0,1)$ and $L>0$
\end{enumerate}
then for all $p\in \{x\in B_{1/4}: u_\e(x)\geq \theta \e\}$ and all $r\geq 2\max(L, M)\e$ such that $B_r(p)\subset B_1$,
\begin{equation}\label{thm:nondeg:maineq}
\sup_{B_r(p)}u_\e \geq c r.
\end{equation}
%for all $p\in \{x\in B_1: u_\e(x)\geq \t \e/4\}$ and $r>2\e/\nu$ such that $B_r(p)\subset B_1$.
\end{theorem}
Assumption (1) above is made to ensure that $u_\e$ satisfies the universal Lipschitz bound in $B_1$. The proof of Theorem \ref{thm:nondeg:main} is achieved in several stages over Section \ref{sec:nondeg}, in which the $\mathcal{D}(\k,L)$ density property hypothesis (2) is first crucially utilized in a Poincar\'e-Sobolev type estimate  (see key Lemma \ref{thm:nondeg}) to get the nondegeneracy growth away from points $p$ in the transition layer $F_\e^\tau$, and later in a limiting argument to extend it for points $p\in \{u_\e> T\e\}$, for small enough $\e>0$. The condition \eqref{eq:condf} that we impose on the nonlinearity $f$, allows us to handle points in the remaining layer $\{\theta \e \leq u_\e\leq \t \e\}\subset Z_\e^\tau$, since it entails that $u_\e$ experiences exponential growth in $Z_\e^\tau$ (see Lemma \ref{lem:nondeg0}). Just as in \cite[Remark 2.3]{AS}, one can relax \eqref{eq:condf} to the assumption $\liminf_{s\downarrow 0}f(s)s^{-p}>0$ for some $p\geq 1$, with virtually no effect on the proof of the theorem (with the only change being that the exponential growth of $u_\e$ in $Z_\e^\t$ is replaced by a polynomial one, leading to a slightly different constant $M$). 

%used in a Poincar\'e-Sobolev type estimate that allows us to bound the $L^1$-norm of $u_\e$ at an interior scale by the \emph{square} of $\max u_\e$ at unit scale. From then on, the result in Theorem \ref{thm:nondeg:main} is achieved via a standard interation argument. 

Our motivation to find conditions under which solutions of the singular perturbation problem \eqref{eq:SLe} enjoy the uniform nondegeneracy property sprang from the recent progress in classifying \emph{global} nonnegative solutions of the semilinear equation \eqref{eq:SLe} for $\e=1$:
\begin{equation}\label{eq:SL}
\left\{
\begin{aligned}
 u & \geq 0 & \quad \mbox{in } \R^n, \\
   \Delta  u &= f(u) & \quad \mbox{in } \R^n,
  \end{aligned}\right.
\end{equation}
and, in particular, the solutions that locally minimize the energy $I_1$. Taking into consideration that their \emph{blow-downs} $u_\e(x):=\e u(x/\e)$ are local minimizers of $I_\e$ that converge to globally defined, energy minimizing, \emph{homogeneous} solutions of the one-phase FBP \eqref{FBP}, Fern\'andez-Real and Ros-Oton formulated a natural conjecture, akin to the celebrated De Giorgi conjecture \cite{DeGio} for the Allen-Cahn equation.
\begin{conj}[\cite{FR-RO2019}]\label{conj:min}
Suppose that $u:\R^n\to (0,\infty)$ minimizes the energy $I_1$ locally. Then for $n\leq n_e^*-1$, $u$ has to be one-dimensional, i.e.\ $u(x)=V(x_n)$ in a suitable Euclidean coordinate system, where $V(t)$ is the unique (positive) solution to the ODE problem
\begin{equation}\label{eq:1d}
V'' = f(V) \quad \text{in } \R, \quad \text{with} \quad V(0)=T \quad \text{and} \quad  V'(0) = 1.
\end{equation}
Here $n^*_e$ denotes the lowest dimension in which there exists a global \emph{singular} homogeneous minimizer of $I_0$. By the works of Caffarelli-Jerison-Kenig \cite{CJK}, Jerison-Savin \cite{JS} and De Silva-Jerison \cite{DSJcone}, it is currently known that $5\leq n^*_e \leq 7. $
\end{conj}
The conjecture was recently established by Audrito and Serra \cite{AS} who devised for the context an ``improvement of flatness" technique inspired by Savin's proof \cite{savin2009regularity} of the De Giorgi conjecture, by bulding upon De Silva's regularity theory method \cite{DS2011} for the one-phase FBP. Their result holds more generally for any critical point $u$ of $I_1$ in any dimension, provided $u$ has an asymptotically flat interface and blows down to $x_n^+$. The Audrito-Serra theorem has since been used by Engelstein, Fern\'andez-Real and Yu \cite{engelstein2022graphical} in proving that global solutions of \eqref{eq:SLe} that are monotone in $x_n$ and satisfy
\begin{equation}\label{eq:monlim}
\lim_{x_n\to-\infty} u(x',x_n) = 0 \quad \text{and}\quad \lim_{x_n\to \infty} u(x',x_n) = \infty,
\end{equation}
have to be one-dimensional in dimensions $n\leq n^*_e$. %Note that monotonicity and \eqref{eq:monlim} together imply that $u$ is in fact a local minimizer of $I_1$ (see the Proof of \cite[Corollary 1.7]{engelstein2022graphical}).

There is another, stronger version of Conjecture \ref{conj:min} that concerns more broadly global \emph{stable} critical points of $I_1$.
\begin{conj}[\cite{FR-RO2019}]\label{conj:stable}
Suppose that $u:\R^n \to (0,\infty)$ is a stable critical point of $I_1$ in $\O=\R^n$, i.e.\ the second variation of $I_1$ at $u$
\begin{equation}\label{eq:clastab}
I_1''(u,\O)[\phi]:= \left.\frac{d^2}{dt^2}\right|_{t=0} I_1(u + t\phi, \O) = 2\int_\O \left(|\n \phi|^2 + f'(u)\phi^2 \right) dx \geq 0 \quad \text{for all } \phi \in C^\infty_c(\O).
\end{equation}
Then for $n\leq 4$, $u(x)=V(x_n)$ in an appropriate Euclidean coordinate system, where $V$ is the solution of \eqref{eq:1d}.  
\end{conj}
The rigidity statement in Conjecture \ref{conj:stable} is currently known to be true only for $n=2$ \cite{FarinaValdi2009}. What seems to make this version more challenging (if one is to employ the blow-down strategy) is a lack of understanding if  blow-down limits of $u$ even solve the one-phase FBP \eqref{FBP} in certain weak sense, let alone what notion of stability is preserved in the limit. To start, it is not known if the strong nondegeneracy property \eqref{eq:nondeg0} holds for stable solutions of \eqref{eq:SL}. The implementation of the strategy is further impeded by the possibility of wedge blow-down limits $s|x_n|$, $s\in (0,1]$, which afflict the study of the rigidity problem for global stable solutions of the one-phase FBP \eqref{FBP} itself (see \cite{KW-2023}). 

In our second main result we prove that if blow-downs of $u$ satisfy a $\mathcal{D}(\k,L)$ density property uniformly, then a weaker  notion of stability is preserved in the limit thanks to the nondegeneracy Theorem \ref{thm:nondeg:main}. This enables the blow-down strategy to be executed, yielding the rigidity result in Conjecture \ref{conj:stable}.  The precise notion of stability that we employ is the one with respect to \emph{compact domain deformations}. 

\begin{definition}\label{def:vars}
Let $\O\subseteq \R^n$ be a domain and let $X\in C^\infty_c(\O;\R^n)$ be a smooth, compactly supported vector field. Denote by $\phi:\R\times \O\to\O$ its associated flow in $\O$:
$$
\de_t \phi_t(x) = X(\phi_t(x)), \quad \phi_0(x) = x. 
$$
The first and second \emph{inner variations} of the functional $I_\e(\cdot, \O)$, $\e\geq 0$, at $u\in H^1_\text{loc}(\O)$, along the vector field $X$ are given respectively by
\begin{align*}
\d I_\e(u,\O)[X]:= \left.\frac{d}{dt}\right|_{t=0} I_\e(u(\phi_t^{-1}(x)), \O) \quad\text{and} \quad \d^2  I_\e(u,\O)[X]:= \left.\frac{d^2}{dt^2}\right|_{t=0} I_\e(u(\phi_t^{-1}(x)), \O).
\end{align*}
A critical point $u_\e$ of $I_\e(\cdot, \O)$ is \emph{stable} with respect to compact domain deformations if
$$
\d^2 I_\e(u,\O)[X]\geq 0 \quad \text{for all } X\in C^\infty_c(\O;\R^n).
$$
\end{definition}
Note that if $u$ is a positive, stable critical point of $I_1(\cdot, \O)$ (in the sense of \eqref{eq:clastab}), then it is also stable with respect to compact domain deformations, since  (see Proposition \ref{prop:S1S2}) 
$$\d^2 I_1(u, \O)[X] = I_1''(u, \O)[\langle \n u, X\rangle] \quad \forall X\in C^\infty_c(\O;\R^n). $$ 
The grace of the stability notion in Definition \ref{def:vars} is that it behaves very well under taking limits. This becomes manifest from the succinct formulas \eqref{eq:firstvar}-\eqref{eq:secvar} that we derive for the first and second inner variations of $I_\e$, which also hold for $\e=0$ (and, in fact, apply to general potentials $\F_\e$ inside the functional $I_\e$). The formulas appear in a different (albeit longer) form already in \cite{le2011second} in the context of the Allen-Cahn equation, but here we derive them with the apparatus of differential geometry which, we insist, provides the right conceptual framework for the calculations (see Appendix \ref{sec:derivationvars}). In this way, we produce formulas (Proposition \ref{prop:vars}) that are valid for general Riemannian manifolds. 

We can now state our second main result. 
\begin{theorem}\label{thm:mainres}
Let $u\in C^\infty(\R^n)$ be a global positive solution of \eqref{eq:SL} that is stable with respect to compact domain deformations. Assume further that there exist constants $\k\in (0,1]$ and $L>0$ such that 
\begin{equation}\label{thm:mainresult:dens}
|B_R(x)\cap \{u\leq \tau/4\}| \geq \k |B_R(x)| \quad \text{for all } x\in \{\tau \leq u \leq T\} \text{ and all large } R\geq L.
\end{equation}
Let $n^*$ be the critical dimension from Definition \ref{def:critdim} (explained also below), which satisfies $$5 \leq n^*\leq n^*_e\leq 7.$$ 
If $n\leq n^*-1$, then in appropriate Euclidean coordinates, $u(x)=V(x_n)$, where $V(t)$ is given by the solution of \eqref{eq:1d}. 
\end{theorem}
We establish Theorem \ref{thm:mainres} by showing that $u$ has asymptotically flat interface in dimensions $n\leq n^* -1$ and then invoking the result of Audrito and Serra \cite[Theorem 1.4]{AS}. In order to prove the asymptotic flatness of the interface, we build a general theory of convergence of solutions $u_\e$ to \eqref{eq:SLe} in a domain $\O\subseteq \R^n$, which are stable with respect to compact domain deformations and satisfy the density property $\mathcal{D}(\k,L)$, as $\e \downarrow 0$. The nondegeneracy result in our first Theorem \ref{thm:nondeg} ensures the Hausdorff distance convergence of the interface of $u_\e$ to the free boundary of the limit $u_0$, as well as the key convergence 
$$
\F_\e(u_\e) \to \F_0(u_0) \quad \text{in } L^1_\text{loc}(\O) \quad \text{as } \e\downarrow 0.
$$
The latter, along with the known $H^1_\text{loc}$-convergence of $u_\e$ to $u_0$, permits stability (in the sense of Definition \ref{def:vars}) to be preserved in the limit.

To encapsulate all the good properties of the limiting function $u_0$, we employ the notion of  \emph{inner-stable solution} to the one-phase FBP (see Definition \ref{def:stablevar} below), introduced recently in \cite{Velichkovetal2022}. This type of weak stable solution of \eqref{FBP} shares much of the same regularity theory as minimizers of the Alt-Caffarelli functional $I_0$. In particular, the free boundary $F_0(u_0)$ is a smooth hypersurface, except possibly on a closed singular subset of Hausdorff dimension at most $n-n^*$, where $n^*$ is precisely the lowest dimension which admits a singular homogeneous inner-stable solution. Since local minimizers of $I_0$ are inner-stable solutions themselves (see Remark \ref{rem:minarestable}), one trivially has $n^*\leq n^*_e$. The lower bound  $n^*\geq 5$ was proved in \cite{Velichkovetal2022} by showing that the nonnegative second inner variation condition implies the \emph{stability inequality} of Caffarelli-Jerison-Kenig \cite{CJK} for free boundary cones $u_0$ with an isolated singularity at the origin:
\begin{equation}\label{eq:CJK}
\int_{F_0(u_0)} H \phi^2 \, d\H^{n-1} \leq \int_{\O_0^+(u)} |\n \phi|^2 \quad \text{for all } \phi \in C^\infty_c(\R^n\setminus\{0\});
\end{equation}
here $H$ denotes the mean curvature of $F_0(u_0)$ with respect to the outer unit normal to $\de \O_0^+(u_0)$. It is only the partial information that energy minimizing cones with an isolated singularity satisfy \eqref{eq:CJK}, in conjunction with the dimension reduction argument of Weiss \cite{weiss1998partial}, that is used in \cite{JS} to obtain the bound $n^*_e\geq 5$.  Given that the inner-stable solution class also admits a dimension reduction principle, the same lower bound holds for $n^*$. 

In Section \ref{sec:stablevar} of our paper we present the regularity theory of inner-stable solutions to the one-phase FBP, developed by \cite{Velichkovetal2022}, almost in its entirety. The reason is two-fold. First, we do it for the reader's convenience, and second -- because we are naturally guided to use our elementary formula \eqref{eq:secvar} for the second inner variation $\d^2 I_0$, which is aligned with the (formal) convergence of $\d^2 I_\e$ to $\d^2 I_0$, in lieu of their more sophisticated formula \cite[(7.8)-(7.9)]{Velichkovetal2022}. The tools of differential geometry allow us to perform the computations leading to the stability inequality \eqref{eq:CJK}  in a transparent, methodical fashion, and in fact, we show that for any test vector field \mbox{$X\in C^\infty_c(\O;\R^n)$} that avoids the singular part of the free boundary $F_0(u_0)$ of a one-phase FBP solution $u_0$, the second inner variation of the Alt-Caffarelli energy $I_0$ at $u_0$ along $X$ has the representation formula (see Proposition \ref{prop:2varsmooth}):
\begin{equation}\label{eq:secvarfb0}
\frac{1}{2}\d^2 I_0(u_0,\O)[X] = \int_{\O^+_0(u_0)} |\n (L_X u_0)|^2\, dx - \int_{F_0(u_0)} H (L_X u_0)^2 \, d\H^{n-1}, 
\end{equation}
where $L_X$ denotes the \emph{Lie derivative} along $X$ (which coincides with the directional derivative when applied to functions). Since $\n u\neq 0$ in $\O^+_0(u)$ for homogeneous one-phase FBP solutions $u$, the condition $\d^2 I_0(u)[X]\geq 0$ is equivalent to  \eqref{eq:CJK}.% by taking $X=\phi |\n u|{-2} \n u$ %Therefore, $n^*\geq 5$ by \cite{JS} and global inner-stable homogenous solutions of \eqref{FBP} in dimensions $n< n^*$ have flat free boundaries. 

The paper is organized as follows. In Section \ref{sec:nondeg} we prove several nondegenerecy estimates for solutions of \eqref{eq:SLe}, which ultimately lead to the proof of Theorem \ref{thm:nondeg:main}. In Section \ref{sec:limsol} we state the formulas \eqref{eq:firstvar}-\eqref{eq:secvar} for the first and second inner variations of the energies $I_\e$, $\e\geq 0$, in the Euclidean setting. We then %define the concept of \emph{inner-stationary} weak solutions to \eqref{eq:SLe} and \eqref{FBP},
build a convergence theory for solutions $u_\e$ of \eqref{eq:SLe}, which satisfy the strong nondegeneracy property \eqref{eq:nondeg0} \emph{for all} $\theta \in (0,T]$. In particular, we show that their limits have trivial first inner variation $\d I_0$.  Section \ref{sec:stablevar} describes the regularity theory of inner-stable solutions to the one-phase FBP, developed by \cite{Velichkovetal2022}. Finally, in Section \ref{sec:proofmainres2} we show that solutions $u_\e$ of \eqref{eq:SLe} that are stable with respect to compact domain deformations and satisfy a $\mathcal{D}(\k,L)$ property uniformly for all small $\e>0$, converge to inner-stable solutions of \eqref{FBP} as $\e \downarrow 0$. As a corollary, we obtain the proof of Theorem \ref{thm:mainres}. 

In the two appendices \ref{sec:derivationvars} and \ref{sec:fbvars} to the article, we provide the key technical results related to the first and second inner variation for the functionals $I_\e$, $\e\geq 0$, which support the exposition in Sections \ref{sec:limsol}-\ref{sec:stablevar}. Appendix \ref{sec:derivationvars} is devoted to the computation of $\d I_\e$ and $\d^2 I_\e$ in the setting of a general Riemannian manifold (see Proposition \ref{prop:vars}); its reading requires a very basic acquaintance with tensor calculus. In it we also expand on the divergence structure of the integrands appearing in the integral formulas for the inner variations (see Lemmas \ref{lem:V1} and \ref{lem:V2}). This latter information is then exploited in Appendix \ref{sec:fbvars}, in which we simplify the formulas for $\d I_0$ and $\d^2 I_0$ in the Euclidean setting and establish the formula \eqref{eq:secvarfb0} for $\d^2 I_0(u_0)$ at a critical point $u_0$ of the Alt-Caffarelli energy. 

With great pleasure I dedicate this paper to David Jerison on the occasion of his 70th birthday. I am profoundly grateful for all the math that I have learned and continue to learn from him, for his generosity, guidance and friendship.

\section{Nondegeneracy estimates}\label{sec:nondeg}

The goal of this section is to establish Theorem \ref{thm:nondeg:main}, which we do in a sequence of nondeneracy estimates. Before we start with these, we record the uniform interior Lipschitz bound that solutions of \eqref{eq:SLe} satisfy. 

\begin{prop}[Uniform Lipschitz continuity; see Theorem 1.2 of \cite{CafSalsa}]\label{prop:Lip} Let $u_\e\in C^2(B_2)$ be a  solution of \eqref{eq:SLe} in $B_2$ and assume that $0\in \{u_\e\leq T\e\}$. Then
\begin{equation}\label{prop:Lip:eq}
\|\n u_\e \|_{L^\infty(B_1)}\leq C
\end{equation}
for some constant $C=C(n,f)>0$.
\end{prop}

We also recall the notation set earlier in \eqref{eq:defZF}:
\begin{equation*}
Z_\e^\theta(u_\e):=\{x\in \O: u_\e(x)\leq \theta\e \} \quad \text{and} \quad F_\e^\theta(u_\e):= \{x\in \O: \theta\e\leq u_\e(x)\leq T\e\}
\end{equation*}
from which we will often drop the reference to $u_\e$, whenever it is implicit. 

The first nondegeneracy lemma can be viewed as the statement that solutions $u_\e$ of \eqref{eq:SLe} experience (exponential) growth inside the set $Z_\e^\t(u_\e)$.

\begin{lemma}\label{lem:nondeg0}
Let $u_\e\in C^2(B_1)$ be a solution of \eqref{eq:SLe} in $B_1\subset \R^n$ and assume that
\[
u_\e(0) \geq \theta \e \quad \text{for some } \theta \in (0,\tau).
\]
Then there exists a constant $c_1=c_1(n)$ such that for $M:=2\log(c_1\tau/\theta)/\sqrt{c_0} = M(\theta, n, c_0)$
\begin{equation}
\sup_{B_{M\e}(0)} u_\e \geq \tau \e, \quad \text{provided } B_{M\e}\subset B_1.
\end{equation}
\end{lemma}
\begin{proof}
Let $B_R(0)$ be the largest ball, centered at the origin, such that $B_R\subseteq Z_\e^\tau$. We would like to show that $R\leq M \e$. We notice that
\[
\D u_\e = f_\e (u_\e) \geq \frac{c_0}{\e^2} u_\e \quad \text{in } B_R,
\]
and thus $v(x):=u_\e(\e x/\sqrt{c_0})/\e$ solves
\[
-\D v + v \leq 0 \quad \text{in } B_{R_1},
\]
where $R_1:=R \sqrt{c_0}/\e$. Defining $w\in C^2(B_{R_1})\cap C(\overline{B_{R_1}})$ to be the solution of $-\D w + w = 0$ in $B_{R_1}$ with boundary values given by $v$, the maximum principle tells us that $v\leq w$ in $B_{R_1}$. Now, it is known (see \cite[pp.\ 214]{CafCordo2}) that $w$ satisfies the weighted mean-value formula
\begin{align*}
w(0) &= \frac{1}{\phi(r)} \dashint_{\de B_r} w \, d\H^{n-1}, \quad \text{for all } r\leq R_1, \\ 
\text{where} \quad  \phi(r) &:= \dashint_{\de B_r} e^{x_1} \, d\H^{n-1}\geq c_1^{-1} e^{r/2},
\end{align*}
for some $c_1=c_1(n)$. Therefore, we have
\[
\theta\leq v(0)\leq w(0)\leq \frac{1}{\phi(R_1)} \dashint_{\de B_{R_1}} v \, d\H^{n-1} \leq c_1 e^{-R_1/2} \tau,
\]
and we can conclude the desired bound 
$$
R= (R_1/\sqrt{c_0})\e\leq (2\log(c_1\tau/\theta)/\sqrt{c_0})\e = M\e.
$$
\end{proof}

For the next nondegeneracy result we will need the following Poincar\'e-Sobolev inequality, whose proof can be adapted from \cite[Theorem 1 on pp.\ 290]{evans2010partial}:
\begin{lemma}\label{lem:PS}
Assume that $g\in W^{1,1}(B_R)$ satisfies 
\[
|x\in B_R: g(x)=0|\geq \k |B_R|.
\]
Then there exists a constant $C=C(\k)$ such that 
\begin{equation}\label{lem:PS:eq}
\|g\|_{L^1(B_R)}\leq C R \|\n g\|_{L^1(B_R)}. 
\end{equation}
\end{lemma}

We now present our key uniform nondegeneracy lemma.
\begin{lemma}\label{thm:nondeg} Let $u_\e\in C^2(B_1)\cap C(\overline{B_1})$ be a solution of \eqref{eq:SLe} in $B_1\subset \R^n$, $\e>0$. Suppose that 
\begin{equation}\label{thm:nondeg:cond2}
u_\e(0) \geq \t\e,
\end{equation}
\begin{equation}\label{thm:nondeg:cond1}
|\{ u_\e \leq (\tau/4)\e \} \cap B_{1/2}| \geq \k |B_{1/2}|, \quad \text{for some } \k\in (0,1],
\end{equation}
and that $u_\e$ satisfies the universal Lipschitz bound \eqref{prop:Lip:eq} in $B_1$. Then 
\begin{equation}\label{thm:nondeg:eq}
\sup_{B_{1}} u_\e \geq \mu,
\end{equation}
for some constant $\mu=\mu(n, \kappa, f)>0$.
\end{lemma}

% Explain the hypotheses. \\

\begin{proof}[Proof of Theorem \ref{thm:nondeg}]
Denote by 
\[
\s:= \sup_{B_{1}} u_\e.
\]
%Let $\rho\in(0,1/4)$ and let $\phi=\phi_\rho\in C^{1}_c(B_1)$ be a standard nonnegative cut-off function, such that $\phi\equiv 1$ in $B_{1-\rho/2}$ and $\|\n \phi\|_{L^\infty(B_1)}\leq 4/\rho$.
We will carry out the proof in several steps. In what follows, the letters $C, c$ (possibly with indices and primes) will denote positive constants which depend only on $n$, $\k$, and $f$.  Take $\rho \in (0,1/4]$. \\
%Because of \eqref{thm:nondeg:cond2}, it suffices to establish the result for $\e\in (0, \e_0)$ where $\e_0=\e_0(n,\d,\k, f)$ is an appropriately small positive constant, whose size will be determined in the process of the proof.

\noindent \emph{Step 1.} We start with the simple estimate
\begin{equation}\label{thm:nondeg:Step1}
\int_{B_{1-\rho}} \D u_\e \, dx \leq C_1 \s \rho^{-2}.
\end{equation}
Indeed, taking a standard, nonnegative cut-off function $\phi\in C^{2}_c(B_1)$ such that $\phi\equiv 1$ in $B_{1-\rho}$ and $\|\phi\|_{C^2(B_1)}\leq c/\rho^2$, we have
\begin{align*}
\int_{B_{1-\rho}} \D u_\e \, dx \leq \int_{B_{1}} \D u_\e \phi \, dx =  \int_{B_{1}} u_\e \D \phi \, dx \leq C_1 \s \rho^{-2}.
\end{align*}

\noindent \emph{Step 2.} We will next show that
\begin{equation}\label{thm:nondeg:Step2}
|B_{1-\rho}\cap \{u_\e > \t \e/2 \}| \leq C_2 \s \rho^{-2}.
\end{equation}
by exploiting the observation that 
\[
1\sim \mathcal{F}_\e(u) \sim \mathcal{F}_\e(u) - \mathcal{F}_\e(\t\e /4) \quad \text{whenever}\quad  u\geq \t\e/2,
\]
with constants depending only on $f$. 
For the purpose, consider $g := (\mathcal{F}_\e(u_\e) - \mathcal{F}_\e(\t \e/4))^+ $ and observe that %$g\in C^{0,1}(\overline{B_{1}})$ with
\[
|\nabla g| \leq |\n\F_\e(u_\e)|  = 2 \D u_\e |\n u_\e| \leq C \D u_\e \quad \text{in } B_1,
\]
because of the assumed universal Lipschitz bound of $u_\e$ in $B_1$. 
Furthermore, because of \eqref{thm:nondeg:cond1}, $g$ vanishes inside $B_{1-\rho}$ on a set of measure at least $\k B_{1/2}$, so that we may apply the Poincar\'e-Sobolev inequality \eqref{lem:PS:eq} to $g$ in $\O=B_{1-\rho}$, obtaining
\begin{equation}\label{thm:nondeg:interm}
\int_{B_{1-\rho}} g \, dx  \leq c \int_{B_{1-\rho}} |\n g|\, dx\leq \tilde{c}  \int_{B_{1-\rho}} \D u_\e |\n u_\e|  \, dx \leq \tilde{c}_1 \s \rho^{-2},
\end{equation}
where the last inequality is a consequence of the bound \eqref{thm:nondeg:Step1} from Step 1. Now, we get
\begin{equation*}%\label{thm:nondeg:Step2pen}
|B_{1-\rho}\cap \{u_\e > \t \e/2 \}| = \int_{B_{1-\rho}\cap \{u_\e > \t \e/2 \}} 1\, dx \leq  \int_{B_{1-\rho}} c_1 (\mathcal{F}_\e(u_\e) - \mathcal{F}_\e(\t \e/4))^+ \, dx \leq  C_2  \s \rho^{-2}.
\end{equation*}

\noindent \emph{Step 3.} At this stage, we will obtain an $L^1(B_{1-\rho})$ bound on $u_\e$ in terms of the \emph{square} of $\sigma=\sup_{B_1}u_\e$:
\begin{equation}\label{thm:nondeg:Step3main}
\int_{B_{1-\rho}} u_\e \, dx \leq C_3 \s^2 \rho^{-2}. 
\end{equation}
First, we claim that 
\begin{equation}\label{thm:nondeg:Step3}
\int_{B_{1-\rho}} u_\e \, dx \leq 2 \int_{\{u_\e > \t \e /2\}} u_\e \, dx. 
\end{equation}
Indeed, we have
\begin{align*}
\int_{B_{1-\rho}} u_\e \, dx &= \int_{B_{1-\rho}\cap \{u_\e \leq u(0)/2\}} u_\e \, dx + \int_{B_{1-\rho}\cap \{u_\e > u(0)/2\}} u_\e \, dx \\
& \leq \frac{1}{2} u_\e(0) |B_{1-\rho}| + \int_{B_{1-\rho}\cap \{u_\e > u(0)/2\}} u_\e \, dx \\
& \leq \frac{1}{2} \int_{B_{1-\rho}} u_\e \, dx +  \int_{B_{1-\rho}\cap \{u_\e > u(0)/2\}} u_\e\, dx 
\end{align*}
where the last inequality follows from the mean-value property, enjoyed by the subharmonic $u_\e$. As $u_\e(0)\geq \t\e$, we confirm the  validity of \eqref{thm:nondeg:Step3}:
\[
\frac{1}{2}\int_{B_{1-\rho}} u_\e \, dx \leq  \int_{B_{1-\rho}\cap \{u_\e > u(0)/2\}} u_\e\, dx \leq  \int_{B_{1-\rho}\cap \{u_\e > \t\e/2\}} u_\e.
\]
Now, \eqref{thm:nondeg:Step3main} follows after combining \eqref{thm:nondeg:Step3} with \eqref{thm:nondeg:Step2}% and using the monotonicity of $\F_\e(u)$:
\begin{align*}
\int_{B_{1-\rho}} u_\e \, dx &\leq 2 \int_{B_{1-\rho}\cap\{u_\e > \t \e /2\}} u_\e\, dx  \leq  2 \s |B_{1-\rho}\cap \{u_\e>\t\e/2\}| \leq  C_3 \s^2 \rho^{-2}.
%2 \int_{B_{1-\rho}\cap \{u_\e > \t \e /2\}} u_\e \frac{\F_\e(u_\e)}{\F_\e(\e \t /2)} \, dx  \\
%& \leq \frac{2\s}{\F(\t/2)} \int_{B_{1-\rho}\cap \{u_\e > \t \e /2\}} \F_\e(u_\e) \, dx \leq C_3 \s^2 \rho^{-2}. 
\end{align*}

\noindent \emph{Step 4.} The $L^1$-estimate of the subharmonic $u_\e$ in $B_{1-\rho}$ entails a bound on the supremum of $u_\e$ on a slightly smaller scale. Indeed, since $u_\e$ is subharmonic, the function
\begin{equation}
r \rightarrow r^{1-n} \int_{\de B_r} u_\e \, d\H^{n-1} 
\end{equation}
is increasing in $r$, so that for $\rho\in (0,1/4]$
\begin{align}
\int_{\de B_{1-3\rho/2}} u_\e \, d\H^{n-1}  & \leq \frac{1}{\rho/2} \int_{1-3\rho/2}^{1-\rho} \left(\frac{1-3\rho/2}{r}\right)^{n-1}  \int_{\de B_r} u_\e \, d\H^{n-1}  \, dr \notag \\
& \leq \frac{2^{n}}{\rho}  \int_{B_{1-\rho}\setminus B_{1-3\rho/2}} u_\e \, dx  \leq c \rho^{-1} \int_{B_{1-\rho}} u_\e\, dx. \label{thm:nondeg:Step4-1}
\end{align}
Furthermore, if $h$ is the harmonic function in $B_{1-3\rho/2}$ whose boundary values on $\de B_{1-3\rho/2}$ are given by $u_\e$, we can estimate $\sup_{B_{1-2\rho}} u_\e$ via the maximum principle and the Poisson representation formula in $B_{1-3\rho/2}$: 
\begin{align}\label{thm:nondeg:Step4-2}
\sup_{B_{1-2\rho}} u_\e \leq \sup_{B_{1-2\rho}} h \leq c \rho^{1-n} \int_{\de B_{1-3\rho/2}} u_\e \, d\H^{n-1}. % \leq c_1 \rho^{-n} \int_{B_{1-\rho}\setminus B_{1-3\rho/2}} u_\e \, dx  \leq c_1 \rho^{-n} \int_{B_{1-\rho}} u_\e\, dx.
\end{align}
Now, the combination of \eqref{thm:nondeg:Step4-2}, \eqref{thm:nondeg:Step4-1} and the estimate \eqref{thm:nondeg:Step3main} from Step 3 yields
\begin{equation}\label{thm:nondeg:Step4-3}
\sup_{B_{1-2\rho}} u_\e \leq C_4 \rho^{-n-1} \s^2. 
\end{equation}

\noindent \emph{Step 5.} In this ultimate step we perform a standard iteration that produces a contradiction if $\s = \sup_{B_1} u_\e$ is too small. Denote
\[
\s_r:= \frac{\sup_{B_r} u_\e}{r}.
\]
The final estimate of Step 4 implies that for $\rho\in (0,1/4]$
\begin{equation}\label{thm:nondeg:fund}
\s_{1-2\rho}  = (1-2\rho)^{-1} \sup_{B_{1-2\rho}} u_\e \leq C_5 \rho^{-n-1} \s_1^{2}.
\end{equation}
Since for $r\in (0,1)$ the blow-up 
$$
\tilde{u}_{\e/r}(x):= u_\e(rx)/r, \quad \text{for } x\in B_1,
$$
is a nonnegative solution of $\D u = f_{\e/r}(u)$ in $B_1$ and satisfies the hypotheses \eqref{thm:nondeg:cond2}-\eqref{thm:nondeg:cond1}, under which \eqref{thm:nondeg:fund} was derived, we obtain, after rescaling, that for $0<r< R \leq 1$,
\begin{equation*}
\sigma_{r} \leq C_5 \left(\frac{R-r}{2R}\right)^{-n-1} \s_R^{2}\leq \tilde{C}_5 (R-r)^{-n-1} \s_R^{2} \quad \text{as long as} \quad 2\rho := \frac{R-r}{R} \leq 1/2.
\end{equation*}
In particular, we have that
\begin{equation}\label{thm:nondeg:itstep}
\sigma_{r} \leq C_6 (R-r)^{-n-1} \s_R^{2} \quad \text{as long as} \quad \frac{1}{2}\leq r<R\leq 1.
\end{equation}
Setting $r_{0}=1$ and defining $r_m=r_{m-1} - 2^{-m-1}$ iteratively for $m\in \N$, we see that $1/2< r_m < r_{m-1} \leq 1$, hence we are allowed to iterate \eqref{thm:nondeg:itstep}:
\begin{equation}\label{thm:nondeg:seq}
\s_{r_{m}} \leq C 2^{m(n+1)} \s_{r_{m-1}}^{2}, \quad m\in \N.
\end{equation}
We claim that \eqref{thm:nondeg:seq} implies that if $\s_1 \leq \mu=\mu(n,\k, f)$ is small enough, then
\begin{equation}\label{thm:nondeg:iter}
\s_{r_m} \leq \s_1 \g^{-m} \quad \text{for } m\in \{0\}\cup \N,
\end{equation}
for some constant $\g=\g(n)>1$. Obviously, \eqref{thm:nondeg:iter} is true for $m=0$, and assume it is true for index $m-1$. Using \eqref{thm:nondeg:seq}, we get that
\begin{align*}
\s_{r_m} \leq C 2^{m(n+1)} \s_1^{2} \g^{-2(m-1)} = (\s_1 \g^{-m})(2^{n+1} \g^{-1})^m (\s_1 C \g^{2}) \leq \s_1 \g^{-m},
\end{align*}
provided we choose $\g = 2^{(n+1)}>1$ and $\s_1\leq \mu=\mu(n,\k, f)$ where $\mu C \g^{2}  = 1$. However, \eqref{thm:nondeg:iter} leads to a contradiction, because for sufficiently large $m$
\[
\sup_{B_{1/2}} u_\e \leq \sup_{B_{r_m}} u_\e  = r_m \s_{r_m} \leq \s_{r_m} < \t \e.
\]
We conclude that $\s_1 > \mu$. 

\end{proof}
As a corollary to Lemma \ref{thm:nondeg}, we get that solutions $u_\e$ of \eqref{eq:SLe} grow linearly away from points of the interface $F_\e^\theta(u_\e)$, whenever $u_\e$ possesses the density property $\mathcal{D}(\k, L)$ of Definition \ref{def:DKL}.

\begin{coro}\label{coro:nondeg} Let $u_\e\in C^2(B_1)\cap C(\overline{B_1})$ be a nonnegative solution of \eqref{eq:SLe} in $B_1$ that satisfies the Lipschitz estimate \eqref{prop:Lip:eq}. Let $\k\in (0,1)$, $L>0$ and $\theta\in (0,\tau]$. If the interface of $u_\e$ has the density property $\mathcal{D}(\k,L)$, then for some positive constants $c=c(n,\k, f)$ and $M=M(n,\theta,f)$,
\begin{equation}\label{coro:nondeg:eq}
\sup_{B_r(p)}u_\e \geq c r \qquad \forall p\in F_\e^\theta(u_\e)  \text{ and }~ \forall r\geq 2\max(L,M)\e, \text{ such that } B_r(p)\subset B_1. 
\end{equation}
\end{coro}
\begin{proof}
Assume that $B_r(p)\subset B_1$ with a radius 
\begin{equation}\label{coro:nondeg:eq0}
r\geq 2 L_0 \e, 
\end{equation}
where $L_0\geq L$. We will analyze the following two cases, in the process of which we will determine the size of $L_0$. 

\noindent\emph{Case 1.} Assume first that we are located at a point $p\in F_\e^\t(u_\e)$. After rescaling at $p$,
\[
\tilde{u}_{\e/r}(x):=u_\e(p+rx)/r \quad \text{for } x\in B_1,
\]
we see that $\tilde{u}_{\e/r}$ satisfies all the hypotheses of Theorem \ref{thm:nondeg}. Therefore, $$\sup_{B_1}\tilde{u}_{\e/r} \geq \mu=\mu(n,\k, f)$$ and thus $\sup_{B_r(p)} u_\e \geq \mu r$.

\noindent\emph{Case 2.} Suppose now that $p\in B_1$ is such that $\theta \e \leq u_\e(p) < \tau \e$. Lemma \ref{lem:nondeg0} informs us that for some $M=M(\theta, n, f)$
\begin{equation}\label{coro:nondeg:eq1}
\sup_{B_{M\e}(p)} u_\e \geq \tau \e.
\end{equation}
Noting that \eqref{coro:nondeg:eq1} says
\[
M\e \leq (M/L_0) r/2,
\]
let us choose $L_0=\max(M,L)$.  In this way, \eqref{coro:nondeg:eq1} implies the existence of a point $\tilde{p}\in \overline{B_{r/2}(p)}$, where $u_\e(\tilde{p})\geq \tau \e$. Applying now the rescaling argument from Case 1 to $u_\e$ in $B_{r/2}(\tilde{p})\subset B_r(p)$ (which is permitted, as $r/2 \geq L\e$ implies that $\tilde{u}_{2\e/r}(x):=u_\e(\tilde{p}+x r/2)/(r/2)$ satisfies the hypotheses of Theorem \ref{thm:nondeg}), we obtain
$$
\sup_{B_r(p)} u_\e \geq \sup_{B_{r/2}(\tilde{p})} u_\e  \geq \mu r/2 = (\mu/2) r. 
$$

%\noindent\emph{Case 3}. Finally, if $r < 2\e/\nu$, then the fact that $u_\e(p)\geq (\tau /4)\e$ implies that
%\[
%\sup_{B_r(p)} u_\e \geq  (\tau /4)\e \geq (\tau \nu/8) r.
%\]

\end{proof}

In the next lemma, we will use the $\mathcal{D}(\k,L)$ property to obtain an important \emph{distance} nondegeneracy estimate in the spirit of \cite[Lemma 5.1]{LedWolanski1998}. We introduce the notation
\begin{equation}
\O^+(u_\e):=\{x\in \O: u_\e(x)>T\e\}
\end{equation}
to denote the $T\e$-superlevel set of $u_\e$ in a domain $\O$.% As before, we may drop the reference to $u_\e$ when it is implicitly understood. 

\begin{lemma}\label{prop:distnondeg} Let $u_\e\in C^2(B_1)$ be a solution of \eqref{eq:SLe} in $B_1\subset \R^n$ and assume that for some positive constants   $C_1, C_2, \k$ and $L$, it satisfies
\begin{itemize}
\item the uniform Lipschitz estimate: $\|\n u_\e\|_{L^\infty(B_1)} \leq C_1$; %of Proposition \ref{prop:Lip};
\item the uniform nondegeneracy condition: $\sup_{B_r(p)} u_\e \geq C_2 r$ whenever $p\in \de \big(B_1^+(u_\e)\big) \cap B_1$, $B_{r}(p)\subseteq B_1$ and $r\geq 2L\e$;
\item the density property $\mathcal{D}(\k, L)$ of Definition \ref{def:DKL}.
\end{itemize}
 Then there exist positive constants $\e_0, \mu_0$, such that for every $\e\leq \e_0$, we have
\begin{equation}
u_\e(y)\geq \mu_0 \, d(y, F_\e^\t) \quad \text{for every } y\in B_{1/2}^+(u_\e), \quad \text{with}\quad d(y, F_\e^\t)\leq 1/4.
\end{equation}
\end{lemma}
\begin{proof}
We will argue by contradiction. Assume that the statement of the proposition is false and we have a sequence of counterexamples $u_{\e_k}$ with $\e_k\downarrow 0$, for each of which there exists a point $y_k\in B_{1/2}^+(u_{\e_k})$ with $d_k:=d(y_k, F_{\e_k}^\t)\leq 1/4$, where 
\begin{equation*}%\label{prop:distnondeg:eq1}
T \e_k \leq u_{\e_k}(y_k)\leq  \frac{1}{k} d_k.
\end{equation*}
Let $z_k\in F_\e(u_\e)\cap \de B_{d_k}(y_k)$ realize the distance $d_k$ between $y_k$ and $F_{\e_k}^\t$. Taking into account that $B_{2 d_k}(z_k) \subset B_{1}$, we may define the rescaled solutions
\[
\tilde{u}_{\e_k/d_k}(x):= u_{\e_k}(z_k+d_k x)/d_k, \quad \text{for } x\in B_2.
\]
Then $v_k:=\tilde{u}_{\e_k/d_k}$ are uniformly Lipschitz continuous in $\overline{B_2}$ and fulfill:
\begin{align}
& v_k(0) = T\e_k/d_k \quad \text{and} \quad \sup_{B_r(0)} v_k \geq C_2 r \quad \text{for all } r\in (2L \e_k/d_k,2); \label{prop:distnondeg:prop1} \\
& v_k \text{ is positive and harmonic in } B_1(q_k), \text{ where } q_k:=(y_k-z_k)/d_k \in \de B_1; \label{prop:distnondeg:prop2} \\ &  v_k(q_k) = u_{\e_k}(y_k)/d_k \leq 1/k; \label{prop:distnondeg:prop3} \\
& |\{v_k\leq (\tau/4) \e_k/d_k\}\cap B_{r}|\geq \k |B_{r}| \quad \text{for all }r\geq 2L(\e_k/d_k). \label{prop:distnondeg:prop4} \
\end{align}
Hence, up to taking subsequences, we can assume that the points $q_k\in \de B_1$ converge to some $q_\infty \in \de B_1$ and that $v_k$ converges uniformly in $\overline{B_2}$ to a Lipschitz continuous function $v_\infty\in C(B_1)$ that is harmonic in its positive phase $\O:=\{x\in B_2: v_\infty(x)>0\}$. Furthermore, as $$\e_\k / d_k \leq 1/(kT) \to 0,$$ we see by \eqref{prop:distnondeg:prop1} that $v_\infty(0)=0$ and that $v_\infty$ is nondegenerate at all scales at $0$: 
$$\sup_{B_r} v_\infty \geq C_2 r, \quad \text{for all }r\in (0,2),$$ 
so that $0\in \de \O$. Because of \eqref{prop:distnondeg:prop2}, we deduce that $v_\infty\geq 0$ is harmonic in $B_1(q_\infty)$. However,  \eqref{prop:distnondeg:prop3} means that $v_\infty(q_\infty) = 0$ and the maximum principle yields $v_\infty\equiv 0$ in $B_1(q_\infty)$. Hence, $B_1(q_\infty)\subseteq \O^c$ is a ball touching $0\in \de \O$ from the zero phase of $v_\infty$. Hence, the asymptotic development result \cite[Lemma 11.17]{CafSalsa} for positive harmonic functions at (left) regular points, in combination with the nondegeneracy of $v_\infty$, entails that for some $\beta >0$,
\begin{equation*}
v_\infty(x) = \beta \langle x, -q_{\infty} \rangle + o(|x|) \quad \text{near } 0, \quad \text{in every nontangential region of } \O.
\end{equation*}
In particular, this means that
\begin{equation}\label{prop:distnondeg:dens}
\lim_{r\to 0}  \frac{|\{v_\infty > 0\}\cap B_r|}{|B_r|} = \frac{1}{2}.
\end{equation}
On the other hand, since for every $\d>0$, the uniform convergence of $v_k$ to $v_\infty$ in $B_2$ implies
\[
\{x\in B_2: v_\infty(x)\leq \d\} \supset \{x\in B_2: v_k(x)\leq (\tau/4) \e_k/d_k\} \cup B_1(q_\infty) \quad \text{for all } k \text{ large enough},
\]
we obtain from the monotone convergence theorem and \eqref{prop:distnondeg:prop4} that 
\begin{equation*}
\liminf_{r\to 0} \frac{|\{v_\infty = 0\}\cap B_r|}{|B_r|} = \liminf_{r\to 0}  \lim_{\d \downarrow 0} \frac{|\{v_\infty  \leq \d\}\cap B_r|}{|B_r|} \geq \k + \frac{1}{2}. 
\end{equation*}
The latter contradicts \eqref{prop:distnondeg:dens} for $\k>0$.
\end{proof}

We are now in a position to establish Theorem \ref{thm:nondeg:main}.
\begin{proof}[Proof of Theorem \ref{thm:nondeg:main}]
As $u_\e(0)\leq T\e$, Proposition \ref{prop:Lip} tells us that the uniform Lipschitz bound \eqref{prop:Lip:eq} holds in $B_1$.  For points $p\in F_\e^\theta(u_\e)\cap B_{1/4}$, the nondegeneracy bound \eqref{thm:nondeg:maineq} thus follows from  \eqref{coro:nondeg:eq} of Corollary \ref{coro:nondeg}. 

Assume now that $p\in B_{1/4}^+(u_\e)=\{x\in B_{1/4}: u_\e(x)>T\e\}$, $r\geq 2L\e$ and $B_{r}(p)\subset B_1$. Since $$d(p, F_\e^\t(u_\e))\leq d(p,0)\leq 1/4,$$ Lemma \ref{prop:distnondeg} states that as long as $\e\leq \e_0=\e_0(n,\k, f, L)$ is small enough, 
\begin{equation}\label{thm:nondeg:main:eq1}
u_\e(p) \geq \mu_0 \, d(p, F_\e^\t),
\end{equation}
for some constant $\mu_0=\mu_0(n,\k,L, f)$. If $r\leq 2 d(p, F_\e^\t)$, then \eqref{thm:nondeg:main:eq1} directly implies that
\[
\sup_{B_r(p)} u_\e \geq u_\e(p) \geq \mu_0 \, d(p, F_\e^\t) \geq (\mu_0/2) r.
\]
In case that $r>2 d(p, F_\e^\t)$, let $\tilde{p}$ be the point in $F_\e^\t\cap B_{1/2}$ that realizes the distance $d(p, F_\e^\t)$. Because we have assumed that $r/2\geq L\e$, we can rescale $u_\e$ in $B_{r/2}(\tilde{p})$ as in the proof of Corollary \ref{coro:nondeg} and apply Theorem \ref{thm:nondeg} to get $\sup_{B_{r/2}(\tilde{p})} u_\e \geq \mu r/2.$
Now, the fact that $r>2 |p-\tilde{p}|$ implies $B_r(p)\supseteq B_{r/2}(\tilde{p})$, so that
\[
\sup_{B_r(p)}u_\e \geq \sup_{B_{r/2}(\tilde{p})} u_\e \geq (\mu/2)r.
\]
\end{proof}

\section{Limits of solutions of \eqref{eq:SLe} as $\e \downarrow 0$}\label{sec:limsol}

We begin this section by recalling the notion of \emph{inner-stationary} solutions of \eqref{eq:SLe}, resp.\ \eqref{FBP}, which are defined as the critical points of $I_\e$ (resp.\ $I_0$) with respect to inner domain deformations.

\begin{definition}\label{def:innerstat} A function $u =u_\e\in H^1_\text{loc}(\O)$ is an \emph{inner-stationary solution} of \eqref{eq:SLe} (resp.\ \eqref{FBP} when $\e=0$) in a domain $\O\subseteq \R^n$  if the first inner variation
\[
\d I_\e(u, \O)[X] = 0 \quad \text{for all } X\in C^\infty_c(\O;\R^n).
\]
\end{definition}

In Proposition \ref{prop:vars} of Appendix \ref{sec:derivationvars} we will compute explicit formulas for the first and second inner variations of $I_\e$, $\e\geq 0$, that hold in the general setting of an oriented Riemannnian manifold. For our Euclidean case they read
\begin{align}
\d I_\e(u, \O)[X] &= \int_\O \left( (|\n u|^2 + \F_\e(u))\text{div}\, X + L_X \bar{\d} (du, du) \right) \,dx; \label{eq:firstvar}\\
\d^2 I_\e(u, \O)[X] &= \int_\O \left( (|\n u|^2 + \F_\e(u))\text{div}(X \text{div}X )+ 2 (\text{div}X) L_X \bar{\d} (du, du) + L_X^2 \bar{\d}(du, du) \right) dx.  \label{eq:secvar}
\end{align}
Here $\bar{\d}$ is the contravariant $(2,0)$-tensor $\bar{\d}=\sum_{ij} \d^{ij} \de_{x_i} \otimes \de_{x_j}$, which gives the Euclidean induced inner product on covectors, and $L_X$ denotes the \emph{Lie derivative}. In standard coordinates, the tensors $L_X \bar{\d}$ and $L_X^2 \bar{\d}$ have components (see the calculations preceding \eqref{ap:eqnsLX1}-\eqref{ap:eqnsLX2}):
\begin{align*}
(L_X \bar{\d})^{ij} &= -(\de_j X^i + \de_i X^j); \\
(L_X^2 \bar{\d})^{ij} &= -X^k \de_k (\de_j X^i + \de_i X^j) + (\de_j X^k + \de_k X^j)\de_k X^i +  (\de_i X^k + \de_k X^i)\de_k X^j,
\end{align*}
where we have adopted the standard summation convention over repeated indices. 

%Note that the integrands in formulas \eqref{eq:firstvar}-\eqref{eq:secvar} are local, so the formulas make sense for any $u\in H^1_\text{loc}(\O)$ on a domain $\O$ that is not necessarily bounded. For this reason, we will take equations \eqref{eq:firstvar}-\eqref{eq:secvar} as defining expressions for the first and second inner variations of $I_\e(\cdot, \O)$ in unbounded domains $\O\subseteq\R^n$.    

It is worth mentioning the well known fact that if $u_\e$ is a classical solution to \eqref{eq:SLe}, then it is also an inner-stationary solution of \eqref{eq:SLe} (see Proposition \ref{prop:crit1crit2}). The benefit of working with these weak solutions is that they behave well under taking limits. The main goal of this section is to establish the convergence result for solutions to the one-phase singular perturbation problem \eqref{eq:SLe}, presented next. Its proof uses classical, well known arguments, with the only novelty being the argument behind the important $L^1_\text{loc}$ convergence $\F_\e(u_\e)\to \F_0(u)$.

\begin{prop}\label{prop:limsol1} Let $\{u_\e\}_\e$ be a family of solutions of \eqref{eq:SLe} in a domain $\O\subset \R^n$, that satisfy
\begin{itemize}
%\item $0 \in F(u_k)$;
\item (Uniform Lipschitz continuity) There exists a constant $C$, such that $\|\nabla u_\e\|_{L^{\infty}(\O)} \leq C$;
\item (Uniform nondegeneracy) For every $\theta\in (0, T]$, there exist positive constants $\e_0$, $c$ and $\l$, such that if $\e\leq \e_0$, then $\sup_{B_r(x)} u_\e \geq cr$ for every $B_r(x) \subseteq
\O$, centered at a point $x\in \{u_\e \geq \theta \e\}$, with $r\geq \l\e$.
\end{itemize}
Then any limit $u \in H^1_{\text{loc}}(\O) \cap C(\O)$ of a uniformly
convergent on compacts sequence $u_k:=u_{\e_k} \to u$, as $\e_k\to 0$, satisfies
\begin{enumerate}[(a)]
\item $u$ is harmonic in $\O^+_0(u)$;
\item $\{u_k\geq \theta \e_k \} \to \overline{\{u>0\}}$ locally in the Hausdorff distance, for all $\theta>0$;
\item $F_{\e_k}^\theta(u_k) \to F_0(u)$ locally in the Hausdorff distance, for all $\theta\in (0, \t]$;
\item $\nabla u_k \to \nabla u$  in $L^2_\text{loc}(\O)$.
\item $\mathcal{F}_{\e_k}(u_{k}) \to \F_0(u) = 1_{\{u>0\}}$ in $L^1_{\text{loc}}(\O)$;
\end{enumerate}
Moreover, $u$ is a Lipschitz continuous, inner-stationary solution of \eqref{FBP} that is nondegenerate:
\begin{equation}\label{prop:limsol1:nondeg}
\sup_{B_r(x)}u \geq \bar{c} r \quad \text{for every } x\in \overline{\O^+_0(u)} \quad \text{and all }  r>0, \text{ such that } B_r(x)\subset \O.
\end{equation}
for some constant $\bar{c}>0$.

\end{prop}
\begin{proof}
The uniform limit $u$ is clearly harmonic in its positive phase $\O_0^+(u)$ and satisfies the same Lipschitz bound: $\|\n u\|_{L^{\infty}(\O)}\leq C$. Let us show that $u$ possesses the nondegeneracy property \eqref{prop:limsol1:nondeg}. Fix $x_0\in \O^+_0(u)$ and $r>0$ such that $B_r(x_0)\subset \O$. Since $u_k(x_0)\to u(x_0)>0$,  we have for all $k$ large enough $u_k(x_0)\geq T\e_k$. Because $r\geq \l(T) \e_k$ for large $k$ as well, the nondegeneracy property of $u_k$ gives us that $\sup_{B_r(x_0)} u_k \geq \bar{c} r$ for all $k$ large enough, with $\bar{c}=c(T)$. Thus, the uniform convergence yields $\sup_{B_r(x_0)} u \geq \bar{c} r$ and we can conclude by continuity that \eqref{prop:limsol1:nondeg} is valid for all points $x$ in the closure $ \overline{\O^+_0(u)}$. In particular, for every $p\in F_0(u)$ and $B_{r}(p)\subset \O$, there exists a point $q\in \overline{B_{r/2}(p)}$ such that $u(q)\geq (\bar{c}/2) r$, so that by the Lipschitz continuity of $u$, the ball $B_{\nu r}(q)\subseteq \O^+_0(u)\cap  B_r(p)$ for $\nu := \min(\bar{c}/(2C), 1/2)$. Hence,
\[
|\O^+_0(u)\cap B_r(p)| \geq \nu^n |B_r| \quad \text{for all } p\in F_0(u) \text{ and } B_{r}(p)\subset \O,
\]
implying that the set of Lebesgue density points of $F_0(u)$ is empty. Therefore,
\begin{equation}\label{prop:limsol1:dens}
|F_0(u)| = 0.
\end{equation}

The proofs of b) and d) are standard and their proofs can be found in \cite[Lemma 5.3]{AS} and \cite[Lemma 3.1]{CafLW97-1}, respectively, so here  we will focus only on proving c) and e). Fix $\d>0$ and choose a compact subset $K\Subset \O$ such that $d(K, \O^c)>\d$. Denote
\begin{align*}
F_k &:= \{x\in K: \theta \e_k \leq u_\e(x) \leq T\e_k\}, &  F^\d_k:=\{x\in \O: d(x,F_k)<\d\}, \\
F_0 &:= F_0(u)\cap K, & F_0^\d:=\{x\in \O: d(x,F_0)<\d\}.
\end{align*}
In order to establish that $F_0\subseteq F_k^\d$, it suffices to show that for every $x\in F_0$, $B_\d(x)\cap F_k\neq \emptyset$ for all $k$ large enough. Assume not: then for all large $k$ (after possibly taking a subsequence) either $B_\d(x)\subset \{u_k< \theta \e_k\}$, or $B_\d(x)\subset \{u_k>T\e_k \}$. In the first case, we will get by uniform convergence that $u\equiv 0$ in $B_\d(x)$, which is impossible as $x\in F_0$. In the second case, Harnack's inequality implies
\[
\sup_{B_{\d/2}(x)} u_k \leq c(n) u_k(x) \to 0 \quad \text{as } k\to\infty,
\]
so that $u\equiv 0$ in $B_{\d/2}(x)$, which is impossible again. 

To prove that $F_k\subseteq F_0^\d$ for all large $k$, assume by contradiction that there is a sequence of points $\{x_k \in F_k\}$, such that $B_\d(x_k)\cap F_0 = \emptyset$. By  taking a subsequence, we may assume that $x_k\to x_\infty \in K$ and 
$$B_{\d/4}(x_k)\subset B_{\d/2}(x_\infty)\subset B_\d(x_k) \quad \text{for all large } k.$$
By possibly taking a further subsequence, it must be the case that either $B_\d(x_k)\subset \{u = 0\}$, or $B_\d(x_k)\subset \{u>0\}$ for all $k$ large enough. The first scenario is impossible, since by the uniform nondegeneracy property of $u_k$, we have
\[
\sup_{B_{\d/2}(x_\infty)} u_k \geq \sup_{B_{\d/4}(x_k)} u_k \geq c(\theta) \d/4,
\]
so that by uniform convergence, $\sup_{B_{\d/2}(x_\infty)} u \geq  c(\theta) \d/4 >0$, contradicting the fact that  $B_{\d/2}(x_\infty) \subset \{u = 0\}.$ The second scenario doesn't occur either, because by the continuity of $u_k$ and the uniform convergence of $u_k\to u$, we would have
\[
u(x_\infty)=\lim_{k\to\infty} u_k(x_k) = 0,
\]
which would contradict the fact that $B_{\d/2}(x_\infty) \subset \{u>0\}$. The proof of c) is completed. 

Let us now treat the claim in e). For the purpose, we will need the following lemma about the relation between Hausdorff convergence and convergence in measure of compact sets.
\begin{lemma}\label{lem:hausconv}
Let $\{K_k\}$ be a sequence of compact subsets of $\R^n$ that converge in the Hausdorff distance to the compact $K_\infty\subset \R^n$. Then
\begin{equation}\label{lem:hausconv:eq}
 \limsup_{k\to\infty} |K_k| \leq |K_\infty|.
\end{equation}
\end{lemma}
\begin{proof}
Fix $\eps>0$ and let $O\supset K_\infty$ be an open set, such that $|O|\leq |K_\infty| + \e$. Because $K_\infty\cap O^c = \emptyset$, the separation between the compact $K_\infty$ and the closed $O^c$, $d(K_\infty, O^c)\geq \d>0$, for some $\d>0$. Hence, $K_\infty^\d:= \{x\in \R^n: d(x, K_\infty)<\d\}$ is disjoint from $O^c$, and by the Hausdorff convergence of $K_k\to K_\infty$, we have that 
$K_k\cap O^c  \subset K_\infty^\d \cap O^c = \emptyset$, i.e. $K_k\subseteq O$ for all large $k$. Thus,
\[
|K_k|\leq O \leq |K_\infty| + \eps. 
\]
Taking the limit superior as $k\to \infty$, and noting that $\eps>0$ is arbitrary, we arrive at \eqref{lem:hausconv:eq}.
\end{proof}
Going back to proving e), we first write
\begin{align}
 \F_{\e_k}(u_k) - \chrc{\{u>0\}} & = \big(F_{\e_k}(u_k)\chrc{u_k\geq T\e_k} - \chrc{\{u>0\}}\big) + F_{\e_k}(u_k)\chrc{\{\theta \e_k \leq u_k < T\e_k \}} + F_{\e_k}(u_k)\chrc{\{u_k<\theta \e_k\}} \notag \\
& =: A_1 + A_2 + A_3. \label{prop:limsol1:repn}
\end{align}
Take an arbitrary compact $K\Subset \O$. Claim that
\begin{equation}\label{prop:limsol1:A1}
\int_K |A_1| \, dx = \int_K |\chrc{\{u_k\geq T\e_k\}} - \chrc{\{u>0\}}| \, dx \to 0 \quad \text{as } k\to\infty,
\end{equation}
i.e.\ that $\chrc{\{u_k\geq T\e_k\}} \to \chrc{\{u>0\}}$ in $L^1_\text{loc}$. Note that if $x\in \{u>0\}$, then $x\in \{u_k\geq T\e_k\}$ for all large $k$, so we must have
\begin{equation}
 \chrc{\{u>0\}} \leq \liminf_{k\to\infty} \chrc{\{u_k\geq T\e_k\}}.
\end{equation}
Fatou's lemma then tells us that 
\begin{equation}\label{prop:limsol1:Fatou}
|K\cap \{u>0\}| = \int_K \chrc{\{u>0\}} \, dx \leq \liminf_{k\to\infty} \int_K \chrc{\{u_k\geq T\e_k\}} \, dx = \liminf_{k\to\infty} |K\cap \{u_k\geq T\e_k\}|,
\end{equation}
with equality if and only if \eqref{prop:limsol1:A1} is valid. Furthermore, equality in \eqref{prop:limsol1:Fatou} does hold, because the result of Lemma \ref{lem:hausconv} yields that 
\[
\liminf_{k\to\infty} |K\cap \{u_k\geq T\e_k\}|\leq \limsup_{k\to\infty} |K\cap \{u_k\geq T\e_k\}| \leq |K\cap \overline{\{u>0\}}| = |K\cap \{u>0\}| ,
\]
on account of the Hausdorff convergence of $K\cap \{u_k\geq T\e_k\} \to K\cap \overline{\{u>0\}}$ from a) plus the fact \eqref{prop:limsol1:dens} that $|F_0(u)|=0$. 

To show that the integrals over $K$ of $|A_2+A_3|$ in \eqref{prop:limsol1:repn} go to $0$ as $k\to\infty$, fix $\d>0$ arbitrary and choose $\theta>0$ small enough such that $\F(u)\leq \d$ for $u\leq \theta$. In this way,
\begin{equation}\label{prop:limsol1:A3}
\int_K |A_3| \, dx \leq \int_K \d \chrc{\{u_k\leq \theta\e_k\}} \,dx \leq |K| \d \quad \text{for all } k.
\end{equation}
Now, to bound the integral of $|A_2|$ over $K$, we will use the fact from b) that $\{\theta \e_k \leq u_k < T\e_k \}\} \to F_0(u)\cap K$ in the Hausdorff distance.  As a result, Lemma \ref{lem:hausconv} implies that for all large enough $k$,
\begin{equation}\label{prop:limsol1:thininter}
|\{\theta \e_k \leq u_k < T\e_k \}\cap K| \leq |F_0(u)\cap K| + \d \leq \d,
\end{equation}
since $|F_0(u)\cap K|=0$. Therefore, for all large $k$, we have
\begin{equation}\label{prop:limsol1:A2}
\int_K |A_2| \, dx \leq |\{\theta \e_k \leq u_k < T\e_k \}\cap K| \leq \d.
\end{equation}
Combining \eqref{prop:limsol1:A1}, \eqref{prop:limsol1:A2} and \eqref{prop:limsol1:A3} and taking $\d\to 0$, we can complete the proof of $e)$. 

Finally, that $u=\lim_{k\to \infty} u_k$ is an inner-stationary solution of the one-phase FBP \eqref{FBP}, is a result of the strong convergences in d) and e):
\begin{align*}
\lim_{\k\to\infty} \d I_{\e_k}(u_k)[X] & = \lim_{k\to \infty}\int_\O (|\n u_k|^2 + \F_{\e_k}(u_k)) \div X + L_X \bar{\d}(du_k, du_k) \, dx \\
& =  \int_\O (|\n u|^2 + \F_{0}(u)) \div X + L_X \bar{\d}(du, du) \, dx = \d I_0(u)[X] \quad \text{for all } X\in C^\infty_c(\O;\R^n),
\end{align*}
coupled with the fact that $\d I_\e (u_k)=0$, since classical solutions $u_k=u_{\e_k}\in C^2(\O)$ of \eqref{eq:SLe} are inner-stationary by default.
\end{proof}

We end this section by showing that the interface of a minimizer of $I_\e(\cdot, \O)$ actually satisfies a $\mathcal{D}(\k, L)$ density property in $\O$ (cf. Definition \ref{def:DKL}) for some universal constants $\k, L>0$. We place the result here because its proof requires some of the ideas present in the convergence result above. 

\begin{prop}\label{prop:minDKL} Let $u_\e \in H^1(\O)$ be a positive minimizer of $I_\e$ in $\O$, $\e>0$. Then there exist positive constants $\k$ and $L$, depending only on $n$ and $f$, such that the interface of $u_\e$ satisfies the density property $\mathcal{D}(\k,L)$ in $\O$.
\end{prop}
\begin{proof} Let $p\in F_\e^\t(u_\e)$ and assume $B_r(p)\subset \O$. By recentering and rescaling, 
\[
u_\e \to \tilde{u}_{2\e/r}(x):=(r/2)^{-1} u_\e(p+(r/2)x) \quad \text{for } x\in B_2,
\]
it suffices to prove the following statement: there exist absolute constants $\e_0=:1/L$ and $\k>0$ such that if $\e\leq \e_0$ and $u_\e$ is a minimizer of $I_\e$ in $B_2$ with $u_\e(0)\in (\t\e, T\e)$, then 
\begin{equation}\label{prop:minDKL:claim}
|Z_\e^{\t/4}(u_\e)\cap B_{1}| \geq \k |B_{1}|. 
\end{equation}
We remark that $u_\e$ satisfies the universal Lipschitz bound \eqref{prop:Lip:eq} in $B_1$: $\|\n u_\e\|_{L^\infty(B_1)} \leq C$. 

Denote by $h$ the harmonic function in $B_1$ with $h=u_\e$ on $\de B_1$. Since $h$ is a competitor to $u_\e$ in $B_1$, we have $I_\e(u_\e, B_1)\leq I_\e(h,B_1)$, so that
\begin{equation}\label{prop:minDKL:min}
\int_{B_1} |\n (u_\e -h)|^2\, dx = \int_{B_1}\left(|\n u_\e|^2 - |\n h|^2) \right) dx \leq \int_{B_1} \left(\F_\e(h)-\F_\e(u_\e)\right) dx,
\end{equation}
where the first equality follows from the harmonicity of $h$. By the Poincare-Sobolev inequality, we then get
\begin{equation}\label{prop:minDKL:next}
c \int_{B_1} (h-u_\e)^2 \, dx \leq \int_{B_1} |\n (u_\e -h)|^2\, dx \leq \int_{B_1} \left(\F_\e(h)-\F_\e(u_\e)\right) dx
\end{equation}
for a dimensional constant $c>0$. Taking into consideration that minimizers of $I_\e$ satisfy the nondegeneracy property \eqref{eq:nondeg0} (see \cite[Lemma 4.2]{AS}), we have $\max_{\de B_1} u_\e \geq c_1$ for some absolute positive constant $c_1=c_1(n,f)$. In combination with the Lipschitz bound, this implies that for some positive constant $c_2=c_2(n,f)$
\begin{equation}\label{prop:minDKL:nondeg}
c_2 \leq  \dashint_{\de B_1} u_\e \, d\H^{n-1} = \dashint_{\de B_1} h\, d\H^{n-1},
\end{equation}
so that the mean-value property and the Harnack inequality for harmonic functions entail
\begin{equation}
h(x)\geq \tilde{c} h(0)= \tilde{c} \, \dashint_{\de B_1} h\, d\H^{n-1} \geq c_3 \quad \text{in } B_{1/2},
\end{equation}
for some $c_3=c_3(n,f).$ On the other hand, from the Lipschitz bound we know that $u_\e\leq T\e + C r$ in $B_r$ for all $r\in(0,1)$, so that for $\e\leq \e_1:=c_3/(4T)$ and $r_0=\min(c_3/(4C),1/2)$, we have $u_\e \leq c_3/2$ in $B_\rho$. Hence, $h-u_\e\geq c_3 - c_3/2 = c_3/2$ in $B_{r_0}$, and \eqref{prop:minDKL:next} gives
\begin{equation}\label{prop:minDKL:lower1}
\int_{B_1} \left(\F_\e(h)-\F_\e(u_\e)\right) dx \geq c \int_{B_{r_0}} (h-u_\e)^2 \, dx \geq c_4, 
\end{equation}
for some $c_4=c_4(n,f)>0$ and all small $\e\leq \e_1$. Furthermore, we get from \eqref{prop:minDKL:nondeg} and the Harnack inequality that 
\begin{equation}\label{prop:minDKL:harnack2}
\inf_{B_r} h \geq \frac{1-r}{(1+r)^{n-1}} \dashint_{\de B_1} h\, d\H^{n-1} \geq c_3 (1-r).
\end{equation}
Now, if $\rho= \rho(n,f)$ is small enough, \eqref{prop:minDKL:lower1} plus the fact that $\F_\e(h)-\F_\e(u_\e)\leq 1$ yield for $\e\leq \e_1$
\begin{equation}\label{prop:minDKL:inter2}
\int_{B_{1-\rho}} \left(\F_\e(h)-\F_\e(u_\e)\right) dx \geq c_3 - |B_1\setminus B_{1-\rho}| \geq c_3/2:=c_5.
\end{equation}
Hence, if $\e\leq \min(\e_1, \e_2)$, where $\e_2=\e_2(n,f)$ is defined by $T \e_2= c_3(1-\rho)$, we obtain from \eqref{prop:minDKL:harnack2} that $h\geq T\e$ in $B_{1-\rho}$, so that \eqref{prop:minDKL:inter2} becomes
\begin{equation}\label{prop:minDKL:inter3}
\int_{B_{1-\rho}} \left(1-\F_\e(u_\e)\right) dx \geq c_4 \quad \text{whenever } \e\leq  \min(\e_1, \e_2).
\end{equation}
Writing the integral on the left-hand side of \eqref{prop:minDKL:inter3} as
\begin{align*}
\int_{B_{1-\rho}} \left(1-\F_\e(u_\e)\right) dx & = \int_{B_{1-\rho}\cap Z_\e^{\tau/4}}  \left(1-\F_\e(u_\e)\right) dx +  \int_{B_{1-\rho}\cap F_\e^{\tau/4}}  \left(1-\F_\e(u_\e)\right) dx \\
&\leq (1-\F(\tau/4)) |B_{1-\rho}\cap Z_\e^{\tau/4}| + |B_{1-\rho} \cap F_\e^{\tau/4}|,
\end{align*}
we see that the claim \eqref{prop:minDKL:claim} will be established for $\k := c_4/[2 (1- \F(\tau/4))]$ and some $\e_0=\e_0(n,f)\leq \min\{\e_1,\e_2\}$, once we show that 
\begin{equation}\label{prop:minDKL:thinout}
|B_{1-\rho} \cap F_\e^{\tau/4}| \to 0 \quad \text{as } \e \downarrow 0. 
\end{equation}
Now, the ``thinning out" of the interface $F_\e^{\tau/4}$, expressed in \eqref{prop:minDKL:thinout}, is a consequence of the uniform nondegeneracy property and can be established in the same way as in the proof of Proposition \ref{prop:limsol1} (see \eqref{prop:limsol1:thininter} above).
\end{proof}

\section{Inner-stable solutions of the one-phase FBP}\label{sec:stablevar}

In this section we present the regularity theory, developed by Buttazo et al.\ \cite{Velichkovetal2022}, for a class of weak solutions of \eqref{FBP}, which carries a notion of stability and which is closed under taking locally uniform limits. 

\begin{definition}\label{def:stablevar} Let $C, c, \k$ be positive real numbers, and let $\O\subseteq \R^n$ be a domain.  We will say that a nonnegative function $u\in H^1_\text{loc}(\O)\cap C(\overline{\O})$ belongs to the class $\mathcal{S}(C, c, \k;\O)$ if the following are satisfied:
\begin{enumerate}
\item $u$ is an inner-stationary solution of \eqref{FBP} in $\O$: $$\d I_0(u, \O)=0,$$ that is harmonic in its positive phase $\O^+_0(u)$;
\item the second \emph{inner} variation of $I_0$ at $u$ is nonnegative:
\[
\d^2 I_0(u,\O)[X]\geq 0 \quad \text{for all test vector fields } X\in C^\infty_c(\O; \R^n);
\] 

\item $u$ is Lipschitz continuous in $\O$ with a Lipschitz constant bounded by $C$:
\[
\|\n u\|_{L^\infty(\O)}\leq C;
\]
\item $u$ is nondegenerate in $\O$ with a nondegeneracy constant $c$:
$$\sup_{B_r(x)} u \geq c r \quad \text{for every } x\in \overline{\O^+_0(u)} \text{ and all balls } B_r(x)\subseteq \O; $$ 
\item the zero phase has \emph{positive density} at least $\k$:
$$|\{u=0\}\cap B_r(x)|\geq  \k|B_r| \quad \text{for all } x\in F_0(u) \text{ and } B_r(x)\subseteq \O.$$
\end{enumerate}
We will say that $u$ is a \emph{inner-stable solution} to the one-phase FBP \eqref{FBP} in $\O$ if $u\in \mathcal{S}(C,c,\k;\O)$ for some positive constants $C, c,$ and $\k$. 
\end{definition} 

\begin{remark}\label{rem:minarestable} Local minimizers $u\in \H^1_\text{loc}(\tilde{\O})$ of the Alt-Caffarelli functional $I_0(\cdot, \tilde{\O})$ with $0\in F_0(u)$ are inner-stable solutions in any domain $\O\Subset \tilde{\O}$. They are known to satisfy properties (3)-(5) (see \cite{AC}). To check that they satisfy (1)-(2) as well, we simply note that if $\phi_t$ denotes the flow along a test vector field $X\in C^\infty_c(\O;\R^n)$, then $u_t(x):=u(\phi_t^{-1}(x))$ is a competitor to $u$ in $\O$ for all $t\in \R$, so that $I_0(u_t,\O)\geq I_0(u, \O)$. As $u_0 = u$, we have
\[
\d I_0(u, \O)[X] = \left.\frac{d}{dt}\right|_{t=0} I_0(u_t, \O) = 0 \quad \text{and} \quad  \d^2 I_0(u, \O)[X] = \left.\frac{d^2}{dt^2}\right|_{t=0} I_0(u_t, \O) \geq 0. 
\]
\end{remark}

The goal of this section is to show that inner-stable solutions to the one-phase FBP enjoy virtually the same regularity theory as local minimizers of the Alt-Caffarelli functional. Namely, we will present the proof of the following theorem.

\begin{theorem}[\cite{Velichkovetal2022}] \label{thm:stablereg} 
Let $u$ be an inner-stable solution of \eqref{FBP} in a domain $\O\subseteq \R^n$. Then its free boundary $F_0(u)$ is a smooth hypersurface, except possibly on a closed singular subset of Hausdorff dimension at most $n-n^*$, where the \emph{critical dimension} $n^*$ is given in Definition \ref{def:critdim} below, and satisfies $5\leq n^*\leq 7$. %The function $u$ is $C^\infty$ up to the regular points of $F_0(u)$.
\end{theorem}

We will first collect some basic results necessary for the blow-up analysis behind Theorem \ref{thm:stablereg}. We start with the fact that the class $\mathcal{S}(C,c,\k;\O)$ is compact in the uniform (on compacts) topology. 

\begin{prop}\label{prop:uniform}
Let $\{u_k\}_k$ be a sequence in $\mathcal{S}(C,c,\k;\O)$ with $0\in F_0(u_k)$ for every $k\in \N$. Then, up to taking a subsequence, $u_k$ converges uniformly on compact subsets to some $u\in \mathcal{S}(C,c,\k;\O)$. Moreover, the subsequence can be taken so that 
\begin{equation}\label{prop:uniform:haus}
\O^+_0(u_k) \to \O^+_0(u_k) \quad \text{and} \quad F_0(u_k)\to F_0(u) \quad \text{locally in the Hausdorff distance}.  
\end{equation}
\end{prop}
\begin{proof}
The uniform Lipschitz continuity, in combination with $u_k(0)=0$, implies the uniform local boundedness of the sequence. Thus, by Arzela-Ascoli, $\{u_k\}$ subconverges on compacts to a continuous function $u$ that satisfies $u(0)=0$ and the same Lipschitz bound $\|\n u\|_{L^\infty(\O)}\leq C$. Moreover, $u\lfloor_{\O^+_0(u)}$ is harmonic as the uniform limit of the harmonic functions $u_k\lfloor_{\O^+_0(u)}$. That $u$ inherits (4)-(5) is straightforward to verify. 

Now, it is well known classically (see \cite[Lemma 1.21]{CafSalsa}) that the uniform Lipschitz continuity and the uniform nondegeneracy of the sequence imply the Hausdorff distance convergence \eqref{prop:uniform:haus}, as well as the convergences:
\begin{equation}\label{prop:uniform:strong}
\n u_k \to \n u  \text{ a.e.\ in }\O \quad \text{and} \quad \chrc{\O^+_0(u_k)}\to \chrc{\O^+_0(u)} \text{ in } L^1_\text{loc}(\O).
\end{equation}
These, in turn, entail that for any test vector field $X\in C^\infty_c(\O;\R^n)$,
\[
\d I_0(u,\O)[X] = \lim_{k\to\infty} \d I_0(u_k,\O)[X] = 0 \quad \text{and} \quad \d^2I_0(u,\O)[X] = \lim_{k\to\infty} \d^2 I_0(u_k,\O)[X] \geq 0,
\]
i.e.\ $u$ inherits the variational properties (1) and (2), as well. 
\end{proof}

Note that all the properties (1)-(5) of Definition \ref{def:stablevar} are scale invariant. Thus, if $u\in \mathcal{S}(C,c,\k;\O)$, then its rescale $u_r(x):=r^{-1}u(rx)$ belongs to $\mathcal{S}(C,c,\k, r^{-1} \O)$. As a corollary to Proposition \ref{prop:uniform}, we see that both blow-up and blow-down limits of solutions in the class $\mathcal{S}$ remain inner-stable solutions.

\begin{coro}\label{coro:blow} Let $\O\subseteq \R^n$ be a domain and let $u\in \mathcal{S}(C,c,\k;\O)$ for some positive constants $C, c, \k$. Assuming that $x_0\in F_0(u)$, then
\begin{enumerate}[(a)]
\item For every sequence $r_k\downarrow 0$, the blow-ups $u_{r_k}(x):=(r_k)^{-1}u(x_0+r_k x)$ subconverge on compact subsets of $\R^n$ to some $u_0 \in \mathcal{S}(C,c,\k; \R^n).$
\item If $\O=\R^n$, then for every sequence $r_k\uparrow \infty$, the blow-downs $u_{r_k}$ subconverge uniformly on compact subsets of $\R^n$ to some $u_\infty \in \mathcal{S}(C,c,\k; \R^n)$.
\end{enumerate}
Moreover, the blow-up limit $u_0$ and the blow-down limit $u_\infty$ are homogeneous functions of degree $1$. 
\end{coro}
\begin{proof}
The claims in (a) and (b) follow from Proposition \ref{prop:uniform}. That the limits $u_0$ and $u_\infty$ are homogeneous of degree one is a consequence of the the Weiss Monotonicity Formula (\cite{weiss1998partial}), which applies to inner-stationary solutions of \eqref{FBP}. 
\end{proof}

Next, we state the notion of viscosity solution to the one-phase FBP (\cite{CafI}, \cite{CafSalsa}) and show that, in fact, inner-stable solutions are viscosity solutions, as well. 

\begin{definition}\label{def:visco} A nonnegative function $u\in C(\O)$ is a viscosity solution of \eqref{FBP} if $u$ is harmonic in $\O_0^+(u)$ and 
\begin{enumerate}
\item (supersolution property) for every $x_0 \in F(u)$ with a tangent ball $B$ from the positive side ($x_0 \in \de B$ and $B\subset \O^+_0(u)$), there is $\a \leq 1$ such that 
\begin{equation}\label{def:visco:super}
u(x) = \a \langle x-x_0, \nu\rangle^+ + o(|x-x_0|)
\end{equation}
as $x\to x_0$ non-tangentially in $B$, with $\nu$ the inner normal to $\de B$ at $x_0$;
\item (subsolution property) for every $x_0 \in F(u)$ with a tangent ball $B$ from the zero side ($x_0 \in \de B$ and $B\subset Z_0(u)$), there is $\b \geq 1$ such that 
\begin{equation}\label{def:visco:sub}
u(x) = \b \langle x-x_0, \nu\rangle^+ + o(|x-x_0|)
\end{equation}
as $x\to x_0$ non-tangentially in $B^c$, with $\nu$ the outer normal to $\de B$ at $x_0$.
\end{enumerate}
\end{definition}

\begin{lemma}\label{lem:visco} Let $u$ be an inner-stable solution of \eqref{FBP} in a domain $\O\subseteq \R^n$. Then $u$ is a viscosity solution of \eqref{FBP} in $\O$.  
\end{lemma}
\begin{proof} 
We will provide the proof of the supersolution property of $u$; the proof of the subsolution property is analogous. 

If $F(u)$ has a tangent ball $B$ from the positive side at $x_0$, then by \cite[Lemma 11.17]{CafSalsa} \eqref{def:visco:super} is satisfied from some $\a > 0$. According to Corollary \ref{coro:blow}, any blow-up limit $u_0$ of $u$ at $x_0$, is an inner-stable solution which is homogeneous of degree $1$. Therefore,
\[
u_0(x) =  \a \langle x, \nu\rangle \quad \text{in } P^+:=\{x\in \R^n: \langle x,\nu \rangle>0 \}.
\]
If $\tilde{\O}:=\O^+(u_0)\setminus P^+ = \emptyset$, then $u_0(x) =   \a \langle x, \nu\rangle ^+$ in all of $\R^n$, so that $F_0(u_0)$ is regular everywhere. By Proposition \ref{prop:1varsmooth}, we then get that $\a=1$.

If $\O_1\neq \emptyset$, we notice that in the spherical section $\tilde{\O}_\S:=\tilde{\O}\cap \S^{n-1}$
\[
-\D_{\S^{n-1}} u = (n-1) u \quad \text{in } \tilde{\O}_\S, \quad \text{and } u=0 \text{ on } \de\tilde{\O}_\S, 
\]
i.e.\ $u\lfloor_{\tilde{\O}_\S}$ is the first Dirichlet eigenfunction of $-\D_{S^{n-1}}$ in $\tilde{\O}_\S$, with associated eigenvalue $(n-1)$. Since, the half-sphere has the same first Dirichlet eigenvalue and contains $\tilde{\O}\cap \S^{n-1}$, then $\tilde{\O}\cap \S^{n-1}$ is a half-sphere, and $u_0(x) =   \a \langle x, \nu\rangle ^+ + \tilde{\a} \langle x, \nu\rangle ^-$ for some $\tilde{\a}>0$. This, however, is inconsistent with the positive density of $Z(u_0)$.
\end{proof}

\begin{definition} Let $u$ be an inner-stable solution of \eqref{FBP} in $\O$. A point $x_0\in F_0(u)$ is called \emph{regular} if $u$ has a blow-up limit at $x_0$ of the form $u_0(x) = \langle x, \nu\rangle^+$ for some unit vector $\nu \in \R^n$. Otherwise, the point is called \emph{singular}. We will denote by $\text{Reg}(u)$ the subset of all regular points of $F_0(u)$ and by $\text{Sing}(u):=F_0(u)\setminus \text{Reg}(u)$ -- the subset of its singular points.  
\end{definition}

\begin{remark} Let $x_0$ be a regular point of $F_0(u)$ of an inner-stable solution $u$ and  let $u_k(x):=r_k^{-1} u(x_0+ r_k x)$ be a sequence of blow-ups converging to $u_0(x) = \langle x, \nu\rangle$, where we may assume $\nu = e_n$. We note that $u_k$ are viscosity solutions of \eqref{FBP} by the previous Lemma  \ref{lem:visco}. Since by \eqref{prop:uniform:haus} we have $\O^+_0(u_k)\cap B_1 \to \O^+_0(u_0)\cap B_1$ in the Hausdorff distance, then for every  $\d>0$ there is $k$ large enough such that 
\[
 B_1 \cap \{x_n > - \d\} \supseteq \O^+_0(u_k)\cap B_1 \supseteq B_1 \cap \{x_n > \d\},
\]
i.e. the free boundary $F_0(u_k)$ is $\d$-flat. When $\d$ is sufficiently small, the classical regularity result ``Flat $\Rightarrow$ Smooth" of Caffarelli (see \cite{CafI, CafII}) kicks in and yields that $F_0(u_k)$ is a smooth graph in $B_{1/2}$. Therefore, in a neighbourhood $U$ of every regular point, $F_0(u)\cap U$ is a smooth hypersurface, separating positive from zero phase, and $u$ is a classical solution of \eqref{FBP} in $U$. %This argument also suggests that $\text{Reg}(u)$ is a relatively open subset of $F_0(u)$. 
\end{remark}

\begin{definition}\label{def:critdim} Define the critical dimension $n^*$ for inner-stable solutions to the one-phase FBP to be the lowest dimension $n$ for which there exists a global inner-stable solution $u:\R^n\to \R$ that is homogeneous of degree one, with $0\in \text{Sing}(u)$.
\end{definition}

\begin{remark}\label{rem:entstable} Note that if $u\in \mathcal{S}(C,c,\k;\R^n)$ is a global inner-stable solution, then by Corollary \ref{coro:blow} any blow-down limit $u_\infty=\lim_{k\to\infty} u_{R_k}$, $R_k\to\infty$, belongs to $\mathcal{S}(C,c,\k;\R^n)$ and is homogeneous of degree $1$. Therefore, when $n\leq n^*-1$, the fact that $F_0(u_\infty)$ is a smooth hypersurface implies that %$\O^+_0(u_\infty)\cap S^{n-1}$ is a half-sphere, and 
$u_\infty = x_n^+$ in some Euclidean coordinate system. Now, since $u_{R_k} \to x_n^+$ locally uniformly, the free boundary $F(u_\infty)$ is asymptotically flat, i.e.\
\[
 B_{R_k} \cap \{x_n > - \d_k\} \supseteq \O^+_0(u)\cap B_{R_k} \supseteq B_{R_k} \cap \{x_n > \d_k\},
\]
with the aspect ratio $\d_k/R_k\to 0$ as $k\to\infty$. As $u$ is a viscosity solution of \eqref{FBP} as well, we conclude from Caffarelli's theorem that $u(x)=x_n^+$. 
\end{remark}

The existence of a singular entire minimizer of \eqref{FBP} in $\R^7$ that is homogeneous of degree $1$, constructed by De Silva and Jerison (\cite{DSJcone}), and the observation in Remark \ref{rem:minarestable} suggest that $n^*\leq 7$.  Due to works by Caffarelli, Jerison and Kenig \cite{CJK}, and Jerison and Savin \cite{JS}, it is currently known that the lower bound for the critical dimension $n^*_\text{e}$, in the case of \emph{energy minimizing} solutions is $n^*_\text{e}\geq 5$. This was achieved  by proving the following slightly more general result. 
\begin{theorem}[\cite{JS}]\label{thm:JS}
Let $u$ be a homogeneous solution of \eqref{FBP} in $\O=\R^n$, such that $0\in F_0(u)$ and $F_0(u)\setminus \{0\}$ is a smooth cone separating positive from zero phase. Assume further that $u$ satisfies the stability inequality
\begin{equation}\label{thm:JS:CJK}
\int_{\O^+_0(u)} |\n \phi|^2 \, dx - \int_{F_0(u)} H \phi^2\, d\H^{n-1} \geq 0 \quad \text{for all } \phi\in C^\infty_c(\R^n\setminus \{0\}), 
\end{equation}
where $H$ denotes the mean curvature of $F_0(u)$ with respect to the outer unit normal to $\de \O^+_0(u)$.
Then, for $n=2,3,4$,  $F_0(u)$ is a hyperplane and $u(x)=\langle x,\nu \rangle^+$ for some unit vector $\nu\in \R^n$.
\end{theorem}

To obtain that the critical dimension for inner-stable solutions enjoys the same lower bound $n^*\geq 5$, \cite{Velichkovetal2022} prove
\begin{prop}[Proposition 7.12 of \cite{Velichkovetal2022}]\label{prop:critdim} Let $u$ be an inner-stable solution of \eqref{FBP} in $\O=\R^n$ that is homogeneous of degree one, with $\text{Sing}(u)=\{0\}$. Then $u$ satisfies \eqref{thm:JS:CJK}. In particular, $n^*\geq 5$.
\end{prop}
Here we give a different proof of this proposition, which is based on the formula \eqref{eq:secvarfb0} for the second inner variation $\d^2 I_0(u)$ that we derive in Proposition \ref{prop:2varsmooth} of Appendix \ref{sec:fbvars}. 

\begin{proof} Since $u$ is homogeneous of degree one, we have 
$$u(x) =\langle \n u , x \rangle \quad \text{for } x\in \{u>0\}.
$$
In particular, $\n u \neq 0$ in $\{u>0\}$ and for every test function $\phi \in C^\infty_c(\R^n\setminus \{0\})$, we can define a test vector field $X\in C^\infty_c(\R^n\setminus \{0\})$ by letting
\[
X := \frac{\n u}{|\n u|^2} \phi \quad \text{in }  \{u>0\},
\]
and extending it across the smooth hypersurface $F_0(u) \setminus \{0\}$ as a smooth vector field, compactly supported away from $0$. In this way, $L_X u = \langle X, \n u \rangle = \phi$ in $\O^+_0(u)$. Since $u$ is harmonic in $\O^+_0(u)$, smooth up to $\text{Reg}(F_0(u)) = F_0(u)\setminus \{0\}$ and an inner-stable solution to \eqref{FBP}, Proposition \ref{prop:2varsmooth} informs us that
\begin{align*}
0\leq \d^2 I_0(u, \R^n)[X] &= \int_{\{u>0\}} |\n (L_X u)|^2 \, dx - \int_{F_0(u)} H (L_X u)^2 \,  d\H^{n-1} \\
& =  \int_{\{u>0\}} |\n \phi|^2 \, dx - \int_{F_0(u)} H \phi^2 \,  d\H^{n-1},
\end{align*}
i.e.\ the stability inequality of Caffarelli-Jerison-Kenig \eqref{thm:JS:CJK} is satisfied. 
\end{proof}

\begin{proof}[Proof of Theorem \ref{thm:stablereg}]
Given the bounds $5\leq n^*\leq 7$ for the critical dimension $n^*$ established in Proposition \ref{prop:critdim}, the proof of the regularity statement now follows from Federer's classical technique of \emph{dimension reduction}, introduced in the free boundary context by Weiss \cite{weiss1998partial}. See \cite[Section 10]{velichkovfbnotes} for details. 
%*** ingredients, uniform non-degeracy + flat implies smooth.
\end{proof}

\section{Proof of Theorem \ref{thm:mainres}}\label{sec:proofmainres2}

In this last section we will provide the proof of our second main result, Theorem \ref{thm:mainres}, which will be a consequence of the nondegeneracy Theorem \ref{thm:nondeg:main}, the convergence result Proposition \ref{prop:limsol1}, the regularity Theorem \ref{thm:stablereg} for inner-stable solutions, and ultimately, the Audrito-Serra theorem \cite{AS}.

We begin by showing that a sequence of solutions $u_{\e_k}$ of \eqref{eq:SLe}, $k\in \N$, that fulfill a $\mathcal{D}(\k,L)$ property uniformly as $\e_k \downarrow 0$ and are stable with respect to compact domain deformations, subconverges to an inner-stable solution of the one-phase FBP (Definition \ref{def:stablevar}).

\begin{prop}\label{prop:stablenodeglim} Let $\{u_{\e_k}\}_k$ be a sequence of solutions of \eqref{eq:SLe} in $B_{2R}$, with $\e_k\downarrow 0$, such that
\begin{itemize}
\item $u_\e(0)\leq T\e$, 
\item the interface of each $u_{\e_k}$ satisfies the density property $\mathcal{D}(\k,L)$ in $B_{2R}$ for some positive constants $\k$ and $L$;
\item $u_{\e_k}$ has a non-negative second inner variation with respect to $I_\e(\cdot, B_{2R})$: $\d^2 I_\e(u_{\e_k},B_{2R})\geq 0$.
\end{itemize}
Then, up to taking a subsequence, $u_{\e_k}$ converge uniformly in $B_{R/4}$ to a function $u$ that is an inner-stable solution to the one-phase FBP in $B_{R/4}$.% (according to Definition \ref{def:stablevar}).
\end{prop}
\begin{proof}
After, rescaling we may assume that $R=1$. Since $u_k:=u_{\e_k}(0)\leq T\e$, we know by Proposition \ref{prop:Lip}, that $u_k$ are uniformly Lipschitz continuous in $B_1$. Furthermore, the nondegeneracy result of Theorem \ref{thm:nondeg:main} tells us that for each $\theta \in (0,\tau]$, there are positive constants $\e_0$, $c$ and $\lambda := 2\max(L, M)$ such that if $\e_k\leq \e_0$, then
$$
\sup_{B_r(x)} u_{k} \geq cr \quad \text{for all }  x\in \{u_k \geq \theta \e\}\cap B_{1/4}  \text{ and all } r\geq \l \e_k, \text{ such that } B_r(p)\subset B_{1/4}. 
$$
The hypotheses of Proposition \ref{prop:limsol1} are therefore met in $\O:=B_{1/4}$, so we can infer that the sequence $\{u_{k}\}$ subconverges on $B_{1/4}$ to a nonnegative Lipschitz continuous function $u$ that is harmonic in its positive phase $\O_0^+(u)$ and is a non-degenerate inner-stationary solution of \eqref{FBP}. Because the same proposition gives us that $\n u_k \to \n u$ in $L^2(B_{1/4})$ and $\F_{\e_k}(u_k)\to \F_0(u)$ in $L^1(B_{1/4})$, we get that for any fixed test fector field $X\in C^\infty(B_{1/4}, \R^n)$,
\[
\d^2 I_0 (u, B_{1/4})[X] = \lim_{k\to\infty} \d^2 I_{\e_k}(u_{k}, B_{1/4})[X] \geq 0, 
\]
i.e. $u$ is a stable critical point of $I_0$ with respect to compactly supported deformations of $B_{1/4}$. 

We have thus confirmed that $u$ satisfies properties (1)-(4) of Definition \ref{def:stablevar} in $B_{1/4}$. To conclude that $u$ is an inner-stable solution of \eqref{FBP} in $B_{1/4}$, it remains to check the positive density of the zero phase along the the free boundary $F_0(u)$. Let $x\in F_0(u)\cap B_{1/4}$ and $B_r(x)\subseteq B_{1/4}$. By the Hausdorff convergence of the interface of $u_k$ to $F_0(u)$ (statement (c) of Proposition \ref{prop:limsol1}), we know that there exists a point $p\in B_{r/2}(x)$ that belongs to $\{\tau \e_k \leq u_k\leq T\e_k\}\cap B_{1/4}$ for all large $k$. Now, as $u_k$ satisfies the density property $\mathcal{D}(\k, L)$, we have
\begin{equation}\label{prop:stablenodeglim:eq1}
|\{u_k\leq (\tau/4)\e_k\}\cap B_{r/4}(p)| \geq \k |B_{r/4}| \quad \text{as long as} \quad r/2\geq L \e_k.
\end{equation}
Moreover, as $\chrc{B_{r/4}(p)\cap \{u>0\}}\leq \liminf_{k\to\infty} \chrc{B_{r/4}(p)\cap \{u_k>T\e_k\}}$, we get by Fatou's lemma that
\begin{equation}\label{prop:stablenodeglim:eq2}
|\{u=0\}\cap B_{r/4}(p)|\geq \limsup_{k\to\infty} |\{u_k\leq T\e_k\} \cap B_{r/4}(p)|
\end{equation}
Combining \eqref{prop:stablenodeglim:eq1} and \eqref{prop:stablenodeglim:eq2}, we obtain the desired density bound:
\[
|\{u=0\}\cap B_r(x)| \geq |\{u=0\}\cap B_{r/4}(p)| \geq \limsup_{k\to\infty} |B_{r/4}(p)\cap \{u_k\leq (\tau/4) \e_k\}| \geq \frac{\k}{4^n} |B_r|.
\]
\end{proof}

We are now finally in a position to prove Theorem \ref{thm:mainres}.
\begin{proof}[Proof of Theorem \ref{thm:mainres}]
Without loss of generality, assume that $u(0)= T$. Let $R_k\uparrow \infty$ and let $\e_k:=1/R_k$. Consider the blow-downs of $u$ at $0$,
\[
u_{\e_k}(x):=\e_k u(x/\e_k),
\] 
which are solutions of \eqref{eq:SLe} in $\R^n$, that are stable with respect to compact domain deformations. Furthermore, condition \eqref{thm:mainres} says that the interface of each $u_{\e_k}$ satisfies the density property $\mathcal{D}(\k, L)$ in $\R^n$.  Invoking Proposition \ref{prop:stablenodeglim}, we see that $u_{\e_k}$ subconverge uniformly on compact subsets of $\R^n$ to a global inner-stable solution $u_0$ of the one-phase FBP \eqref{FBP}.  %On the other hand, solutions of \eqref{eq:SL} enjoy a Weiss type monotonicity formula (see \cite[Lemma 5.5]{AS}), so that, by a standard argument, any blow-down limit of $u$ is homogeneous of degree one. We conclude that $u_0$ is an inner-stable solution that is homogeneous of degree one, so that by Remark \ref{rem:entstable} we know that $u_0(x)=x_n^+$ in a suitable Euclidean coordinate system. 
Given that $n< n^*$, Remark \ref{rem:entstable} informs us that $u_0$ actually has a flat free boundary and equals $x_n^+$, in an  appropriate Euclidean coordinate system. From the Hausdorff distance convergence result of Proposition \ref{prop:limsol1}, we see that 
\[
\{x_n > -\d_k\}\cap B_{R_k} \supseteq \{u\geq \tau \} \cap B_{R_k} \supset \{u\geq T \} \cap B_{R_k} \supseteq \{x_n> \d_k\}\cap B_{R_k},
\]
with the aspect ratio $\d_k/R_k\to$ as $k\to\infty$. We may thus invoke the rigidity result \cite[Theorem 1.4]{AS} of Audrito and Serra and conlude that $u(x) = V(x_n)$, where $V$ is the solution of \eqref{eq:1d}.
\end{proof}

%\bigskip

%\noindent{\textbf{Disclosure statement.}}\\
%The author reports there are no competing interests to declare.

\appendix

\section{First and second inner variations of $I_\e$ in an oriented Riemannian manifold}\label{sec:derivationvars}

% we assume some basic background in differential geometry and tensor calculus

Let $(M,g)$ be an oriented Riemannian manifold with induced volume form $\vol_g$. In this section we will compute expressions for the first and the second \emph{inner} variations of the functional
\[
I_\e(u,M)=\int_M \left(|\n_g u|_g^2 + \F_\e(u) \right) \, \vol_g, \quad u\in H^1((M,g)), \quad \e\geq 0,
\]
i.e.\ with respect to deformations of $M$, generated by compactly supported vector fields. The norm of the gradient $\n_g u$ is measured with respect to the metric $g$ and we note that
\[
|\n u|_g^2 := g(\n_g u, \n_g u) =  \bar{g} (du, du) =: |du|_g^2,
\]
where $\bar{g}_p$ denotes the induced inner product on covectors $\xi\in T^*_p(M)$. Take $X\in \G(TM)$ a smooth, compactly supported vector field on $M$ and let $\phi:\R\times M\to M$ be its associated flow:
$$
\de_t \phi_t(x) = X(\phi_t(x)), \quad \phi_0(x) = x. 
$$
For all $t\in \R$, $\phi_t: M\to M$ defines a diffeomorphism of $M$ onto itself with inverse $(\phi_t)^{-1}=\phi_{-t}$, generated by $(-X)$. 
Fix a function $u\in H^1_\text{loc}(M)$ and set 
$$u_t(y):=u((\phi_t)^{-1}(y)) = u(\phi_{-t}(y)) = (\phi_{-t})^*u(y),
$$ 
where $(\phi_{-t})^*$ denotes the pullback by $\phi_{-t}$. Then $u_t\in H^1_\text{loc}(M)$ and we are interested in computing
\[
\d I_\e(u,M)[X]:= \left.\frac{d}{dt}\right|_{t=0} I_\e(u_t, D) \quad\text{and} \quad \d^2  I_\e(u,M)[X]:= \left.\frac{d^2}{dt^2}\right|_{t=0} I_\e(u_t, D),
\]
where $D\subset M$ is a compact subset of $M$, containing the support of $X$. 

\begin{prop}\label{prop:vars} Assume the above setup. Then the first and second inner variations of $I_\e(\cdot, M)$ at $u$ along the vector field $X$ are given by 
\[
\d I_\e(u, M)[X] = \int_M V_1(u; X) \vol_g \quad \text{and} \quad  \d^2 I_\e(u, M)[X] = \int_M V_2(u; X) \vol_g, 
\]
where
\begin{align}
V_1(u; X) & :=(|d u|_g^2 + \F_\e(u))\text{div}_g X + [L_X\bar{g}](du,du), \label{ap:firstvar} \\
V_2(u; X) & := (|d u|_g^2 + \F_\e(u))\text{div}_g\left((\text{div}_g X) X\right) + 2 [L_X\bar{g}](du,du) \text{div}_g X + [L_X^2 \bar{g}](du,du), \label{ap:secondvar}
\end{align}
and $L_X$ denotes the $Lie$ derivative along $X$.
\end{prop}

We refer the reader to the book of Lee \cite[Chapter 12]{LeeSM} for a discussion of the many nice properties that the Lie derivative enjoys. We recall that in local coordinates $\{x^i\}$ of $M$, the Lie derivative of a $(2,0)$ tensor field $S = S^{ij} \de_{x^i} \otimes \de_{x^j}$ takes the form
\begin{align*}
(L_X S)^{ij} &= (L_X S)(dx^i, dx^j) = X\cdot S(dx^i, dx^j) - S(L_X d x^i, dx^j) - S(d x^i, L_Xdx^j) \\
&= X^k \de_k S^{ij} - S(d (dx^i(X)), dx^j) - S(dx^i, d (dx^j(X))) \\
& = X^k \de_k S^{ij} - S^{kj} \de_k X^i - S^{ik} \de_k X^j,
\end{align*}
where we have adopted the standard summation convention over repeated indices. For a domain $M=\O\subseteq \R^n$ of Euclidean space, equipped with the Euclidean metric $g=\d$, the expressions for $L_X \bar{\d}$ and $L_X^2 \bar{\d}$ in the standard coordinates then take the form
\begin{align}
(L_X \bar{\d})^{ij} &= -(\de_j X^i + \de_i X^j); \label{ap:eqnsLX1}\\
(L_X^2 \bar{\d})^{ij} &= -X^k \de_k (\de_j X^i + \de_i X^j) + (\de_j X^k + \de_k X^j)\de_k X^i +  (\de_i X^k + \de_k X^i)\de_k X^j. \label{ap:eqnsLX2}
\end{align}
%Hence, the integrands $V_1$ and $V_2$ of \eqref{ap:firstvar} and \eqref{ap:secondvar} calculate to 
%\begin{align}
%V_1(u;X) =~ & (|du|^2 +  \F_\e(u)) \div X - 2 u_i u_j \de_j X^i; \\
%V_2(u;X) =~ & (|du|^2 +  \F_\e(u)) \div(\div X X) - 4 (\div X) u_i u_j \de_j X^i  + \\
%& -2 X^k \de^2_{kj} X^i u_i u_j  + 2 \de_j X^k \de_k X^i u_i u_j + 2 \de_k X^j \de_k X^i u_i u_j, \notag
%\end{align}
%where we make use of the symmetry of the tensors $L_X \d^{-1}$ and  $L_X^2 \d^{-1}$.

\begin{proof}[Proof of Proposition \ref{prop:vars}]
After changing variables, $y=\phi_t(x)$, we get
\begin{align*}
I_\e(u_t,D) &= \int_D |d [\phi_{-t}^*u]|_g^2 + \F_\e(\phi_{-t}^*u) \, \vol_g(y) = \int_D \left(\phi_t^*(|d [\phi_{-t}^*u]|_g^2)+ \F_\e(u(x)) \right) (\phi_t^*\vol_g)(x) \\
&= \int_D (h_t + \F_\e(u)) \nu_t, \quad \text{where } \nu_t:= \phi_t^* \vol_g \quad \text{and} \\
h_t &:= \phi_t^*(|d [\phi_{-t}^*u]|_g^2) = \phi_t^* \left[\bar{g}\left(d\phi_{-t}^*u, d\phi_{-t}^*u \right)\right)].
\end{align*}
Since the differential $d$ commutes with pullbacks, we can rewrite the expression for $h_t$ as:
\begin{equation*}
h_t = \phi_t^* \left[\bar{g}\left(\phi_{-t}^*(du), \phi_{-t}^*(du) \right)\right)].
\end{equation*}
We can view $\bar{g}\in \Gamma(TM\otimes TM)$ as a contravariant $(2,0)$ tensor field and $\bar{g}(\omega_1, \omega_2)$, where $\o_1, \o_2$ are 1-forms, as the corresponding contraction of the (2,2) tensor field $\bar{g}\otimes \omega_1\otimes \omega_2$. Using the fact that pullbacks and contractions commute, and that pullbacks distribute over tensor products, we can further simplify
\begin{equation}\label{ap:h_t}
h_t = \left(\phi_t^*\bar{g}\right)\left(\phi_t^* [\phi_{-t}^*(du)], \phi_t^*[\phi_{-t}^*(du)] \right) = \left(\phi_t^*\bar{g}\right)(du, du)=:\mu_t(du,du), 
\end{equation}
since $\phi_t^* \phi_{-t}^* = (\phi_{-t}\circ \phi_t)^*=\text{id}^* = \text{id}$. In \eqref{ap:h_t}, $\mu_t:=\phi_t^*\bar{g}$ denotes the pullback of the tensor field $\bar{g}$ by $\phi_t$. The $t$-derivatives of the tensor fields $\mu_t$ and $\nu_t$  can now be computed using the celebrated formula \cite[Proposition 12.36]{LeeSM}
\[
\frac{d}{dt} \phi_t^* S = \phi_t^* (L_X S) \quad \text{for any tensor field } S.
\] 
%We recall that $L_X$ is a derivation and takes tensor fields to tensor fields of the same type. 
We obtain
\begin{equation}\label{ap:mu_t}
\frac{d}{dt} \mu_t = \phi_t^* (L_X\bar{g}) \quad \text{and} \quad \frac{d^2}{dt^2} \mu_t =  \phi_t^* (L_X^2 \bar{g}).
\end{equation}
Similarly,
\begin{equation}\label{ap:nu_t}
\frac{d}{dt} \nu_t = \frac{d}{dt}  \phi_t^*\vol_g = \phi_t^* (L_X \vol_g)  \quad \text{and} \quad \frac{d^2}{dt^2} \nu_t =  \phi_t^* (L_X^2 \vol_g).
\end{equation}
It is well known (\cite[pp.\ 425]{LeeSM}) that the Lie derivative of $\vol_g$ computes to 
\begin{equation}\label{ap:Lievol}
L_X \vol_g = (\text{div}_g X) \vol_g,
\end{equation}
and by using the property that $L_X$ is a derivation, we can further calculate
\begin{align}
L_X^2 \vol_g &= L_X (\text{div}_g X) \vol_g + (\text{div}_g X) L_X\vol_g = \left(d (\text{div}_g X)(X) + (\text{div}_g X)^2 \right)\vol_g \notag \\
& = \text{div}_g\left((\text{div}_g X) X\right)\vol_g. \label{ap:Lie2vol}
\end{align}
Based on the preceding observations, we see that $t\to I_\e(u_t, D)$ is a smooth function, whose first derivative at $t=0$ is given by
\begin{align*}
\d I_\e(u)[X] &= \left.\frac{d}{dt}\right|_{t=0}\int_D (\mu_t(du, du) + \F_\e(u)) \nu_t =\int_D \left(\left(\mu_0(du, du) + \F_\e(u)\right)\dot{\nu}_0 + \dot{\mu}_0\nu_0\right)   \\
&= \int_M \left((|d u|_g^2 + \F_\e(u))\text{div}_g X + [L_X\bar{g}](du,du) \right)\vol_g  %:= \int_M V_1 \vol_g
\end{align*}
and whose second derivative at $t=0$ is
\begin{align*}
\d^2 I_\e(u)[X] &= \int_D \left(\mu_0(du, du) + \F_\e(u)\right)\ddot{\nu}_0 + 2 \dot{\mu}_0(du, du)\dot{\nu}_0+ \ddot{\mu}_0\nu_0  \\
& = \int_M \left((|d u|_g^2 + \F_\e(u))\text{div}_g\left((\text{div}_g X) X\right) + 2 [L_X\bar{g}](du,du) \text{div}_g X + [L_X^2 \bar{g}](du,du) \right)\vol_g, 
%& := \int_M V_2\, \vol_g
\end{align*}
according to the computations in \eqref{ap:h_t}--\eqref{ap:Lie2vol}.
\end{proof}

We end this section by fleshing out the divergence structure in the integrands $V_1$ and $V_2$ of \eqref{ap:firstvar} and \eqref{ap:secondvar}. For ease of notation, we will drop subscripts $g$ and denote
\[
e := |du|^2 + \F_\e(u).
\]

\begin{lemma}\label{lem:V1} Assume that $u\in C^2(W)$ and $\F_\e(u)\in C^1(W)$ in an open subset $W\subseteq M$, $\e\geq 0$.  Then 
\begin{equation}\label{ap:V1div}
V_1 = \text{div}\left(e X  - 2 (L_X u)\n u \right) + \left(2\D u - \F_\e'(u)\right)(L_X u) \quad \text{in } W.
\end{equation}
\end{lemma}
\begin{proof}
We compute in $W$: 
\begin{align*}
 (|d u|^2 + \F_\e(u))\text{div} X & = \text{div}\left(e X \right) - L_X |du|^2  - \F_\e'(u) L_X u; \\ 
 [L_X\bar{g}](du,du)  &= L_X |du|^2 - 2 \bar{g}(L_X du, du) = L_X |du|^2 - 2 \bar{g}(d(L_X u), du) = \\
& = L_X |du|^2 - 2 g(\n (L_X u), \n u) = L_X |du|^2 - 2 \text{div}((L_X u)\n u) + 2 \D u (L_X u),
\end{align*} 
where we used the fact that $L_X$ commutes with the differential $d$. Adding the two equalities above, we obtain \eqref{ap:V1div}.
\end{proof}

As an easy corollary, we get the following well known result.
\begin{prop}\label{prop:crit1crit2} Let $\e>0$. If $u\in C^2(M)$ and $\F_\e(u)\in C^1(M)$, then
\[
\d I_\e(u,M)[X] = -I_\e'(u,M)[L_X u]
\]
for all compactly supported, smooth vector fields $X\in \G(TM)$. In particular, if $u\in C^2(M)$ is a positive critical point of $I_\e$, then the first inner variation $\d I_\e(u,M)=0$.
\end{prop}

In the next lemma we provide the divergence structure within $V_2$.
\begin{lemma}\label{lem:V2} Assume that $u\in C^3(W)$ and $\F_\e(u)\in C^1(W)$ in an open subset $W\subseteq M$, $\e\geq 0$. Then we have in $W$:
\begin{align}
V_2 &= \div\, Y - \left((\div X) \F_\e'(u) L_X u + 2 \D u (L_X^2 u) \right) + 2 |d (L_X u)|^2, \quad \text{where}\label{ap:V2div1} \\ 
Y &= \left(e\,  \div X + L_X |du|^2 - 4 g(\n (L_X u), \n u) \right)X + 2 (L_X^2u) \n u. \label{ap:V2div2}
\end{align}
\end{lemma}
\begin{proof}
We manipulate the terms comprising $V_2$ as follows:
\begin{align*}
(1)\quad & (|d u|^2 + \F_\e(u))\text{div}\left(\text{div} X X\right) = \text{div}\left(e \, \text{div} X X\right) - (\div X) L_X |du|^2 - (\div X) \F_\e'(u) L_X u;  \\
(2) \quad & [L_X\bar{g}](du,du) \text{div} X = (\div X) L_X |du|^2 - 2 (\div X) \bar{g}(L_X du, du); \\
%& = (\div X) L_X |du|^2  - 2 \div \left(\bar{g}(L_X(du), du) X \right) + 2 L_X  \bar{g}(L_X(du), du) \\
(3) \quad & [L_X\bar{g}](du,du) \text{div} X  + [L_X^2 \bar{g}](du,du)  =  \\
  &=  [L_X\bar{g}](du,du) \text{div} X +  L_X \left([L_X \bar{g}](du,du)\right) - 2 [L_X \bar{g}](L_X du, du) \\
 &= \div\left([L_X\bar{g}](du,du) X \right) - 2 \left( L_X \left(\bar{g}(L_X du, du)\right) - \bar{g}(L_X^2 du, du) - \bar{g}(L_X du, L_X du) \right)
\end{align*}

Hence, after adding the three equalities,  we obtain
\begin{align*}
V_2 & = \div\, \tilde{Y} -  (\div X) \F_\e'(u) L_X u - 2 \left( (\div X) \bar{g}(L_X du, du) +L_X \left(\bar{g}(L_X du, du)\right) \right) + \notag \\
& \qquad \quad \, \, \, + 2\bar{g}(L_X^2 du, du) + 2 |L_X du|^2  \notag \\
& = \div\, \bar{Y} - 2 \div\left(\bar{g}(L_X(du), du) X \right) + 2\bar{g}(L_X^2 du, du) + 2 |L_X du|^2 \notag\\
& = \div\, \bar{Y} -  (\div X) \F_\e'(u) L_X u + 2 |L_X du|^2 + 2\bar{g}(L_X^2 du, du)\notag \\
& = \div\, \bar{Y} -  (\div X) \F_\e'(u) L_X u + 2 |L_X du|^2 + 2\bar{g}(d (L_X^2 u), du) \notag\\
& = \div\, \bar{Y} -  (\div X) \F_\e'(u) L_X u + 2 |L_X du|^2 + 2g(\n(L_X^2 u), \n u) \notag \\
& = \div\, \bar{Y} -  (\div X) \F_\e'(u) L_X u + 2 |L_X du|^2 + 2 \div \left( L_X^2 u \n u \right) - 2 \D u (L_X^2 u) \notag \\
& = \div\, Y - \left((\div X) \F_\e'(u) L_X u + 2 \D u (L_X^2 u) \right) + 2 |d (L_X u)|^2 
\end{align*}
where 
\begin{align*}
Y &:= e (\text{div} X) X + [L_X\bar{g}](du,du) X  - 2 \bar{g}(L_X(du), du) X + 2 L_X^2u \n u \notag \\
& = \left(e\,  \div X + L_X |du|^2 - 4 g(\n (L_X u), \n u) \right)X + 2 (L_X^2u) \n u. \label{ap:V2div2}
\end{align*}
\end{proof}

\begin{prop}\label{prop:S1S2} Let $u\in C^3(M)$ be a critical point of $I_\e$ such that $f_\e(u)= \frac{1}{2}\F'_\e(u)\in C^1(M)$, $\e>0$. Then
\[
\d^2 I_\e (u, M)[X] = I_\e''(u, M)[L_X u].
\]
\end{prop}
\begin{proof}
Since $u\in C^3(M)$ is a critical point of $I_\e$, we have $2\D u = \F'_\e(u)$.  After integration, the divergence terms in \eqref{ap:V2div1} vanish and we are left with 
\begin{align*}
\d^2 I_\e(u)[X] &= \int_M 2 |d(L_X u)|^2 - 2 \D u (L_X^2 u) - (\div X) \F_\e'(u) L_X u \\
& = \int_M 2 |d(L_X u)|^2 - \F_\e'(u) \left(L_X (L_X u ) + \div X L_X u \right) = \int_M 2 |d(L_X u)|^2 -\F_\e'(u) \div([L_X u] X) \\
& = 2 \int_M |d(L_X u)|^2 - \div\left(f_\e(u) (L_X) u X\right) + f_\e'(u)(L_X u)^2 = I_\e''(u)[L_X u],
\end{align*}
after another application of the Divergence theorem. 
\end{proof}

\section{First and second inner variations for regular free boundaries}\label{sec:fbvars}

We will apply the formulas in Lemma \ref{lem:V1} and \ref{lem:V2} to simplify the expressions for the first and second inner variations of the Alt-Caffarelli energy $I_0$ in the Euclidean setting.  

\begin{definition} Let $W\subset \R^n$ be a open set. We say that a point $p\in \de W$ is $C^1$-\emph{regular} if there exists $r>0$ and a $C^1$ function $g:\R^{n-1}\to \R$ such that in a suitable Euclidean coordinate system
\[
W\cap B_r(p) = \{x=(x', x_n)\in B_r(p): x_n> g(x')\}.
\]
Otherwise, we call p \emph{singular}. We will denote by $\text{Reg}(\de W)$ the (relatively open) subset of $C^1$-regular points of $\de W$. 
\end{definition}

\begin{prop}\label{prop:1varsmooth}
Let $\O\subseteq \R^n$ be a Euclidean domain and assume that $u\in H^1_\text{loc}(\O)\cap C(\O)$ is a nonnegative inner-stationary solution of \eqref{FBP} in $\O$ that satisfies
\begin{itemize}
\item $u$ is harmonic in $\O^+_0(u)=\{x\in \O: u(x)>0\}$;
%\item the positive phase $\O^+_0(u)=\{x\in \O: u(x)>0\}$ has almost $C^1$-boundary;
\item $u$ is $C^1$ up to $\text{Reg}(F_0(u))$.
\end{itemize}
Then $|\n u|(p) = 1$ at every $C^1$-regular point $p\in F_0(u)$. 
\end{prop}
\begin{proof}
Pick a regular point $p\in F_0(u)$ and let $B$ be a small enough ball centered at $p$ such that $W:=\O^+_0(u)\cap B$ is the supergraph of a $C^1$ function. Let $X\in C^\infty_c(B; \R^n)$. Since $u\in C^\infty(W)\cap C^1\left(W\cup \text{Reg}(F_0(u))\right)$ and $\F_0(u)=1$ in $W$,  \eqref{ap:V1div} tells us that
\begin{equation*}
V_1(u, X) = \div ((|\n u|^2 + 1) X - 2 X\cdot u \n u) \quad \text{in }W,
\end{equation*}
as  $\D u = 0$ in $W$. Now, since $\n u = 0$ a.e.\ in $\{u=0\}$, we see that
\[
0=\d I_0(u)[X] =  \int_B V_1(u,X) \, dx = \int_D V_1(u, X) \, dx = \int_{F_0(u) \cap B} \left\langle (|\n u|^2 +1) X - 2 L_X u \n u, \nu \right\rangle  d\H^{n-1}
\]
where the last equality is a consequence of the Divergence Theorem and $\nu$ denotes the outer unit normal to $\de D$. As $\langle 2 L_X u \n u, \nu \rangle = 2|\n u|^2 \langle X, \nu \rangle$, we deduce
\[
0 = \int_{F_0(u) \cap B} (-|\n u|^2 + 1) \langle X, \nu \rangle\, d\H^{n-1}.
\]
Since $X\in C^\infty_c(B;\R^n)$ can be taken arbitrary, we conclude that $|\n u(p)| = 1$.
\end{proof}

%(Classical bootstrap, $Reg$ is smooth)

\begin{prop}\label{prop:2varsmooth} Let $\O\subseteq \R^n$ be a Euclidean domain and suppose that $u\in H^1_\text{loc}(\O)\cap C(\O)$ satisfies
\begin{itemize}
\item $u$ is an inner-stationary solution of \eqref{FBP}: $\d I_0(u,\O)=0$;
\item $u$ is harmonic in $\O^+_0(u)=\{x\in \O: u(x)>0\}$;
%\item the positive phase $\O^+_0(u)=\{x\in \O: u(x)>0\}$ has almost $C^1$-boundary;
\item $u$ is $C^2$ up to the $\text{Reg}(F_0(u))$.
\end{itemize}
Then for every vector field $X\in C^\infty_c(\O, \R^n)$ supported away from the singular part of $F_0(u)$, the second inner variation of $I_0$ at $u$, along $X$, equals
\begin{equation}\label{prop:2varsmooth:eq}
\frac{1}{2}\d^2 I_0(u,\O)[X] = \int_{\O^+_0(u)} |\n (L_X u)|^2\, dx - \int_{\text{Reg}(F_0(u))} H (L_Xu)^2 \, d\H^{n-1}, 
\end{equation}
where $H$ denotes the mean curvature of the regular free boundary $\text{Reg}(F_0(u))$ with respect to the outer unit normal $\nu=-\n u$.  
\end{prop}
\begin{proof}
Since $\n u = 0$ a.e.\ in $\{u=0\}$, the integration in the formula for $\d^2 I_0(u,\O)[X]$ can be taken only over the positive phase $W:=\O^+(u)$. In $W$ $u$ is smooth and $\F_0'(u) = 0$, so that we have the validity of formulas \eqref{ap:V2div1}-\eqref{ap:V2div2}, indicating
\[
V_2(u;X) = \div Y + 2 |\n (L_X u)|^2 \quad \text{in } W,
\]
on account of the fact that $\D u = 0$ in $D$, where $Y$ is given by \eqref{ap:V2div2}. Denote $\Sigma:=\text{Reg}(F_0(u))$. Since $V_2$ is supported away from the singular part of $F_0(u)$, we may apply the Divergence Theorem to obtain
\begin{align}
\d^2 I_0(u,\O)[X] &= \int_D V_2(u;X) \, dx = 2 \int_D |\n (L_X u)|^2 \, dx + \int_{\Sigma} \langle Y_0, \nu \rangle\, d\H^{n-1} \notag \\
& = 2 \int_D |\n (L_X u)|^2 \, dx - \int_{\Sigma} \langle Y_0, \n u \rangle\, d\H^{n-1}, \label{prop:2varsmooth:eq1}
\end{align}
where $Y_0$ is the continuous vector field on $\Sigma$, defined by
\[
Y_0(p) = \lim_{x\to p, x\in D} Y(x),
\]
with $Y(x)$ given by \eqref{ap:V2div2}. Note that in \eqref{prop:2varsmooth:eq1} we have used Proposition \ref{prop:1varsmooth} that the outer unit normal to $\de \Sigma$, $\nu = -\n u$.

We claim that 
\begin{equation}\label{prop:2varsmooth:main}
\frac{1}{2}\langle Y_0, \n u \rangle =  H (L_X u)^2 + \div_\Sigma \left((L_X u )X^T\right) \quad \text{on } \Sigma,
\end{equation}
where $X^T$ denotes the component of $X$ tangential to $\Sigma$, and $\div_\Sigma \, Z$ denotes the surface divergence of a vector field $Z$ on $\Sigma$:
\[
\div_\Sigma Z(x) = \sum_{i=1}^{n-1} \langle D_{e_i} Z(x), e_i \rangle, \quad \text{for an orthonormal basis }\{e_i\}_{i=1}^{n-1} \text{ of } T_x \Sigma.
\]
Once we establish \eqref{prop:2varsmooth:main}, the formula \eqref{prop:2varsmooth:eq} will be a consequence of \eqref{prop:2varsmooth:eq1} and the Divergence Theorem, applied in $\Sigma$. 

Pick any point $p\in \Sigma$. It will be convenient to work in a Euclidean coordinate system  $(x_1, \ldots x_n)$ centered at $p$, such that the unit vector along $x_n$, $e_n=\nabla u(p)$. With this choice, $u_i(p)=0$ for $i\in S:=\{1,2, \ldots, n-1\}$, $u_n(p)=1$ and 
\begin{align*}
|\n u|_i(p) =\frac{\de_{x_i}|\n u|^2}{2|\n u|} = u_j u_{ji} = u_{ni}(p). 
\end{align*}
Since $|\n u|=1$ on $\Sigma$, we have $u_{ni}(p)=0$ for $i\in S$. Furthermore, because of harmonicity and the fact that $|\n u|(p)=1$, the mean curvature of $\Sigma$ with respect to the outer unit normal $\nu = -\n u$,
\begin{equation}\label{prop:2varsmooth:H}
H=  \div\frac{\n u}{|\n u|} = -|\n u|_n = - u_{nn} \quad \text{at } p.
\end{equation}
With all this in mind, let us calculate the left-hand side of \eqref{prop:2varsmooth:main}, using the coordinates above. Since 
\[
e(x):= |\n u|^2 + F_0(u) =|\n u|^2 + 1 \to 2 \quad \text{when } x\to \Sigma,
\]
we have at $p$,
\begin{align}
\frac{1}{2} \langle Y_0, \n u\rangle &=  \div X X^n + \frac{1}{2} X^i \de_{x_i} |\n u|^2 X^n - 2 X^n \de_{x_n} (L_X u) + L_X (L_X u) \notag \\
& =  \div X X^n +  X^i u_{ij} u_j X^n -  2 X^n (L_X u)_n +  \sum_{i\in S} X^i (L_X u)_i  +  X^n (L_X u)_n \notag \\
& =  \div X X^n +  u_{nn} (X^n)^2 +  \sum_{i\in S} X^i (L_X u)_i  -   X^n (X^i u_i)_n \notag\\
& =  \div X X^n +  u_{nn} (X^n)^2 +  L_{X^T} (L_X u) -   X^n \de_i X^i u_i - X^n X^i u_{in}\notag \\
& = \div X X^n  + u_{nn} (X^n)^2 +  L_{X^T} (L_X u) - X^n \de_n X^n - (X^n)^2 u_{nn} \notag \\
& =  X^n \sum_{i\in S} \de_{x_i} X^i + L_{X^T} (L_X u) = (L_X u) \div_\Sigma X +L_{X^T} (L_X u). \label{prop:2varsmooth:LHS}
\end{align}
On the other hand, as $X^T = X - \langle X ,\n u\rangle \n u$ on $\Sigma$, the right-hand side of \eqref{prop:2varsmooth:main} equals
\begin{align}
H (L_X u)^2  &+ \div_\Sigma \left((L_X u )X^T\right) = H (L_X u)^2 + L_{X^T} (L_X u) + (L_X u) \div_{\Sigma} X^T\notag \\
% &= H (L_X u)^2 + L_{X^T} (L_X u) + (L_X u) \div_\Sigma (X - \langle X, \n u\rangle \n u ) \\
& =  H (L_X u)^2 + L_{X^T} (L_X u) + (L_X u) \div_\Sigma (X) - (L_X u) \div_\Sigma (\langle X, \n u\rangle \n u ) \notag \\
& =  L_{X^T} (L_X u) + (L_X u) \div_\Sigma (X) + H (L_X u)^2  - (L_X u) \sum_{i\in S} \langle D_i [(L_X u) \n u], e_i \rangle \notag \\
& = L_{X^T} (L_X u) + (L_X u) \div_\Sigma (X)  + (L_X u)^2 \left(H - \sum_{i\in S} u_{ii} \right) \notag \\
& = L_{X^T} (L_X u) + (L_X u) \div_\Sigma (X) + (L_X u)^2 (-\D u) =  L_{X^T} (L_X u) + (L_X u) \div_\Sigma (X), \label{prop:2varsmooth:RHS}
\end{align}
where we used the harmonicity of $u$ and the formula \eqref{prop:2varsmooth:H} for the mean curvature of $\Sigma$ to obtain the last line. Now, \eqref{prop:2varsmooth:LHS} and \eqref{prop:2varsmooth:RHS} give \eqref{prop:2varsmooth:main}, thereby completing the proof of the proposition.

\end{proof}

\bibliography{SL_bib.bib}
\end{document}